\documentclass[11pt]{amsart}

\usepackage{amssymb,amsmath,amsthm}
\usepackage{bm}
\usepackage{enumerate}
\usepackage{fullpage}
\usepackage{mathtools}

\renewenvironment{proof}[1][]{\noindent{\scshape Proof{#1}. }}{\qed\vspace{0.1cm}}

\theoremstyle{definition}
\newtheorem{theorem}{Theorem}[section]
\newtheorem{lemma}[theorem]{Lemma}
\newtheorem{proposition}[theorem]{Proposition}
\newtheorem{corollary}[theorem]{Corollary}
\newtheorem{definition}[theorem]{Definition}

\newtheorem{alphtheorem}{Theorem}

\theoremstyle{definition}
\newtheorem{example}[theorem]{Example}
\newtheorem{remark}[theorem]{Remark}

\newtheorem{assumption}[theorem]{Assumption}

\newtheorem{conjecture}[theorem]{Conjecture}

\numberwithin{equation}{section}

\newcommand{\bbC}{\mathbb{C}}
\newcommand{\bbR}{\mathbb{R}}
\newcommand{\bbZ}{\mathbb{Z}}

\newcommand{\bfc}{\mathbf{c}}

\newcommand{\calG}{\mathcal{G}}
\newcommand{\calH}{\mathcal{H}}

\newcommand{\calQ}{\mathcal{Q}}
\newcommand{\calS}{\mathcal{S}}

\newcommand{\fraka}{\mathfrak{a}}
\newcommand{\frakh}{\mathfrak{h}}
\newcommand{\frakq}{\mathfrak{q}}
\newcommand{\fraku}{\mathfrak{u}}

\newcommand{\frakH}{\mathfrak{H}}

\newcommand{\ct}{\mathrm{ct}}

\newcommand{\triv}{\mathrm{triv}}
\newcommand{\End}{\mathrm{End}}
\newcommand{\eps}{\varepsilon}

\DeclareMathOperator{\Rad}{Rad}

\title{An elementary approach to non-symmetric shift operators and their $\mathbf{q}$-analogs}

\author{Max van Horssen}
\author{Maarten van Pruijssen}

\date{\today}

\begin{document}


\begin{abstract}
    We give an algebraic construction of shift operators for the non-symmetric Heckman-- Opdam polynomials and the non-symmetric Macdonald--Koornwinder polynomials. To each linear character of the finite Weyl group, we associate forward and backward shift operators, which are differential-reflection and difference-reflection operators that satisfy certain transmutation relations with the (Dunkl--)Cherednik operators. In the Heckman--Opdam case, the construction recovers the non-symmetric shift operators of Opdam and Toledano Laredo for the sign character. Furthermore, in rank one, we recover the rank-one non-symmetric shift operators previously obtained by the authors and Schlösser.
\end{abstract}


\maketitle


\section{Introduction}


\subsection{Background}
Heckman--Opdam \cite{HeOp87} and Macdonald--Koornwinder \cite{Koo92,Mac00} polynomials are far-reaching generalization of symmetric functions and orthogonal polynomials, including Schur functions, Jack polynomials, and Askey--Wilson polynomials. These polynomials have the distinctive feature of being simultaneous eigenfunctions of a commutative algebra of differential (or difference) operators. This property fundamentally connects them to completely integrable systems of Calogero--Moser--Sutherland-type. The polynomials are attached to (pairs of affine) root systems and root multiplicities, accompanied by a deformation parameter $q$ in the Macdonald--Koornwinder case. For special parameters, the polynomials appear as zonal spherical functions on (quantum) compact symmetric pairs, see e.g.~\cite{HeSc94,Let04}.

In 1982, Macdonald \cite{Mac82} formulated his famous constant term conjectures, generalizing earlier conjectures by Dyson and Mehta, and five years later he \cite{Mac00} conjectured explicit formulae for the norms and evaluations of the polynomials. These conjectures had a profound influence on the theory and in subsequent years partial results were established by several authors by relying on the classification of root systems. A uniform resolution of these conjectures was given by Opdam \cite{Opd89} for the classical case, based on his joint work with Heckman on (hypergeometric) shift operators. Cherednik \cite{Che95} generalized these techniques by introducing $q$-analogs of the shift operators, which allowed him to resolve the conjectures in the $q$-case.

Shift operators owe their name to their defining property of introducing a shift in the root multiplicities of the polynomials. Initially, the existence of shift operators was ad hoc for low-rank root systems \cite{Opd88a} and relied in general on transcendental methods \cite{Opd88b}. This changed drastically after Dunkl's discovery of a family of commuting differential-reflection operators in the rational case that were equivariant for the action of the Weyl group. Using a trigonometric analog of Dunkl's operators, Heckman \cite{Hec91} gave an elementary algebraic method of constructing shift operators, as well as the differential operators having the Heckman--Opdam polynomials as eigenfunctions. These Dunkl--Heckman operators remain equivariant for the action of the Weyl group; however, mutual commutativity is no longer preserved.

An alternative trigonometric analog is due to Cherednik \cite{Che91}, known as the Dunkl--Cherednik operators, which preserve the mutual commutativity at the expense of Weyl group equivariance.  With this development, the representation theory of the (affine) Hecke algebra was brought into the picture, extended later by the inclusion of the celebrated double affine Hecke algebra \cite{Che95}. Furthermore, the mutual commutativity of these operators made it possible to study their simultaneous eigenfunctions, providing the non-symmetric analogs of the Heckman--Opdam \cite{Opd95} and Macdonald--Koornwinder \cite{Mac96,Sah99} polynomials.

A natural question is whether analogous shift operators exist for the non-symmetric polynomials. The first to consider these so-called non-symmetric shift operators were Opdam and Toledano Laredo with their announcement of the existence (forward and backward) shift operators for the non-symmetric Heckman--Opdam polynomials. Their approach is transcendental in nature as it relies on a Paley--Wiener-type theorem \cite[Thm.8.6]{Opd95}. In parallel, rank-one non-symmetric shift operators were studied by the authors and Schlösser \cite{vHvP24, vHS25} in relation with shift operators for matrix-valued Jacobi\footnote{(Non-)symmetric Heckman--Opdam polynomials are frequently called (non-)symmetric Jacobi polynomials. We reserve the nomenclature `Jacobi' for the classical orthogonal polynomials and their non-symmetric analogs.} and Askey--Wilson polynomials. These methods were algebraic, albeit decidedly rank-one, and they established the existence of the forward, backward, and contiguity-type non-symmetric shift operators.


\subsection{Main results}
In this article, we develop an elementary algebraic approach of constructing (non-symmetric) shift operators for the non-symmetric Heckman--Opdam polynomials and the Macdonald--Koornwinder polynomials. At its core, the approach relies on the shift principle, in the spirit of Heckman’s construction of the symmetric shift operators.

We proceed to detail the main results. For this, we briefly introduce some essential notation. A comprehensive account of the definitions and notations is given in Sections \ref{sec:root_systems}--\ref{sec:graded_hecke_algebra} (Heckman--Opdam) and Sections \ref{sec:affine_root_systems}--\ref{sec:affine_hecke_algebra} (Macdonald--Koornwinder).

Let $E_\mu(k)$ denote the non-symmetric Heckman--Opdam polynomials associated with the root system $R$ and the root multiplicity $k$, indexed by the weight lattice $P$ of $R$. We view these polynomials as elements of the group algebra $\bbC[P]$ of $P$. The Dunkl--Cherednik operators $T_\xi(k)$ are differential-reflection operators, i.e., operators on $\bbC[P]$ composed of directional derivatives and reflections, which are indexed by the vector space $\frakh$.

\begin{definition}
    A differential-reflection operator $\calS(k)$ is called a non-symmetric shift operator with shift $l$, a root multiplicity, if it satisfies the following transmutation property with the Dunkl--Cherednik operators:
    \[
        \calS(k) T_\xi(k) = T_\xi(k + l) \calS(k), \qquad \xi \in \frakh.
    \]
    As a result of the transmutation property, the operator $\calS(k)$ maps eigenfunctions of $T_\xi(k)$ to eigenfunctions of $T_\xi(k + l)$. Thus, for $\mu \in P$, we have that $\calS(k) E_\mu(k) = \calH(\mu,k) E_\nu(k + l)$ for a certain $\nu \in P$ and a constant $\calH(\mu,k)$, which we will refer to as the shift factor.
\end{definition}

The fundamental shifts $l$ are multiplicities that equal either $1$ or $0$ on certain Weyl orbits in $R$. These multiplicities correspond to linear characters of the Weyl group $W$ of $R$. Explicitly, the multiplicity $l$ associated with the character $\eps$ is given by
\[
    l(\alpha) = 
    \begin{cases}
        1 & \text{if } \eps(r_\alpha) = -1; \\
        0 & \text{otherwise}.
    \end{cases}
\]
As noted in \cite[\S5.5]{Mac03}, if $R$ is simply-laced, e.g., $R$ is of type $A_n$, then there are only two characters: the trivial and sign character, otherwise there are four characters, e.g., when $R$ is of type $BC_n$.

At this point, we can state our main result in the Heckman--Opdam setting.

\begin{alphtheorem}[See Theorem \ref{thm:transmutation_prop}]\label{thm:a}
    For each linear character $\eps$ of $W$, there exist non-symmetric shift operators $\calG^{(\eps)}_\pm(k)$ with shifts $\pm l$, i.e.,
    \[
        \calG^{(\eps)}_\pm(k) T_\xi(k) = T_\xi(k \pm l) \calG^{(\eps)}_\pm(k), \qquad \xi \in \frakh,
    \]
    which are called the $\eps$-forward/backward non-symmetric shift operators. These operators shift the multiplicity of the non-symmetric Heckman--Opdam polynomials by $\pm l$:
    \[
        \calG^{(\eps)}_\pm(k)E_\mu(k) = \calH^{(\eps)}_\pm(\mu,k)E_{\mu_{\eps,\pm}}(k \pm l), \qquad \mu \in P,
    \]
    for certain explicit shift factors $\calH^{(\eps)}_\pm(\mu,k)$, see Corollary \ref{cor:adj_shift_factor}, and see Section \ref{sec:non_sym_shift_operators_1} for the definition of the weights $\mu_{\eps,\pm} \in P$.
\end{alphtheorem}

For the sign character, the operators $\calG^{(\eps)}_\pm(k)$ coincide with the forward/backward non-symmetric shift operators from \cite[\S8]{OTL24}. In the rank-one case, i.e., when $R$ is of type $BC_1$, we recover the forward/backward non-symmetric shift operators from \cite{vHvP24}, see Example \ref{exa:rank_one}.

The construction of the operators $\calG^{(\eps)}_\pm(k)$ is inspired by the construction of a shift operator for the Knizhnik--Zamolodchikov (KZ) equations by Felder and Veselov \cite{FeVe94}. Their approach relies on the Matsuo--Cherednik correspondence \cite{Che94,Mat92}, which establishes a correspondence between ($\eps$-) symmetric Heckman--Opdam polynomials and polynomial solutions of the KZ-equations. These solutions of the KZ-equations are equivalent to the non-symmetric Heckman--Opdam polynomials. Thus, their construction yields a shift operator for these polynomials; however, it is not a differential-reflection operator as it depends on the spectral parameter. A more detailed discussion of this construction and its obstacles is given in Section \ref{sec:kz_equations}. By resolving the obstruction arising from the dependence on the spectral parameter, we establish our algebraic approach of constructing non-symmetric shift operators.

We now turn to our main result in the Macdonald--Koornwinder setting. We retain the notation of $E_\mu(k)$ for the non-symmetric polynomials, indexed by the lattice $L$, which are associated with a pair of affine root systems $(S,S')$, a pair of root multiplicities $(k,k')$, and a deformation parameter~$q$. The polynomials are regarded as elements of the group algebra $K[L]$ of $L$ over a certain subfield $K$ of $\bbR$. The Cherednik operators $Y^{\lambda'}(k)$ are difference-reflection operators, i.e., operators on $K[L]$ composed of difference operators and reflections, which are indexed by the lattice $L'$.

\begin{definition}
    A difference-reflection operator $\calS(k)$ is called a non-symmetric $q$-shift operator with shift $l^\wedge$, a root multiplicity on $S$, if it satisfies the following transmutation property with the Cherednik operators:
    \[
        \calS(k) Y^{\lambda'}(k) = Y^{\lambda'}(k \pm l^\wedge)\calS(k), \qquad \lambda' \in L'.
    \]
\end{definition}

The affine root system $S$ contains a (finite) root system $S_0$ with Weyl group $W_0$. As previously, each linear character of $W_0$ induces a root multiplicity on $S_0$, which extends to a multiplicity $l$ on~$S$. A rescaled version $l^\wedge$ of $l$, see Remark \ref{rmk:shift_mult}, will serve as the fundamental shift in this case.

To conclude this section, we state our main result in the Macdonald--Koornwinder setting.

\begin{alphtheorem}[See Theorem \ref{thm:transmutation_prop_q}]\label{thm:b}
    For each linear character $\eps$ of $W_0$, there exist non-symmetric $q$-shift operators $\calG^{(\eps)}_\pm(k)$ with shifts $\pm l^\wedge$, i.e.,
    \[
        \calG^{(\eps)}_\pm(k) Y^{\lambda'}(k) = Y^{\lambda'}(k \pm l^\wedge)\calG^{(\eps)}_\pm(k), \qquad \lambda' \in L',
    \]
    which are called the $\eps$-forward/backward non-symmetric $q$-shift operators. These operators shift the multiplicity of the non-symmetric Macdonald--Koornwinder polynomials by $\pm l^\wedge$:
    \[
        \calG^{(\eps)}_\pm(k)E_\mu(k) = \calH^{(\eps)}_\pm(\mu,k)E_{\mu_{\eps,\pm}}(k \pm l^\wedge),
    \]
    for certain explicit shift factors $\calH^{(\eps)}_\pm(\mu,k)$, see Corollary \ref{cor:adj_shift_factor_q}.
\end{alphtheorem}

To the best of the authors' knowledge, the operators described in Theorem \ref{thm:b} are the first known instances of shift operators for the non-symmetric Macdonald--Koornwinder polynomials in the higher-rank case. For the rank-one case, the relation to the rank-one non-symmetric $q$-shift operators from \cite{vHS25} is discussed in Remark \ref{rmk:rank_one_q}.\vspace{-0.1cm}


\subsection{Structure}
The article is organized into two parts. The first part, consisting of Section \ref{sec:heckman_opdam_theory} and Section \ref{sec:non_sym_shift_operators}, treats the (non-)symmetric Heckman--Opdam polynomials and their shift operators. The Macdonald--Koornwinder case is addressed in the second part (Section \ref{sec:non_sym_shift_operators_q}).

Sections \ref{sec:root_systems}--\ref{sec:graded_hecke_algebra} serve to introduce the (non-)symmetric Heckman--Opdam polynomials, the Dunkl--Cherednik operators, and the graded Hecke algebra. Heckman's construction of the symmetric shift operators is given in Section \ref{sec:sym_shift_operators}, with emphasis on the shift principle. In Section \ref{sec:kz_equations}, the relation to the KZ-equations is discussed, accompanied by the construction of the KZ-shift operator by Felder and Veselov.

At the heart of this article lies the algebraic construction of the non-symmetric shift operators, detailed in Section \ref{sec:non_sym_shift_operators_1} and Section \ref{sec:non_sym_shift_operators_2}, in which the shift principle again plays a prominent role. Several properties of these operators are discussed in Section \ref{sec:non_sym_shift_operators_prop}, including a conjecture concerning fundamental non-symmetric shift operators, see Conjecture \ref{conj:fundamental_shift_operator}. In Section \ref{sec:l2_norms}, an application to the computation of the $L^2$-norms of non-symmetric Heckman--Opdam polynomials is given.

Finally, in Section \ref{sec:non_sym_shift_operators_q}, we extend the algebraic construction to obtain shift operators for the non-symmetric Macdonald--Koornwinder polynomials. The exposition follows the structure of Section~\ref{sec:heckman_opdam_theory} and Section \ref{sec:non_sym_shift_operators}.\vspace{-0.1cm}


\section{Background on Heckman--Opdam theory}\label{sec:heckman_opdam_theory}\vspace{-0.1cm}
Let us proceed by providing some background on Heckman--Opdam theory, following \cite{HeSc94,Opd00}.\vspace{-0.05cm}


\subsection{Preliminaries on root systems}\label{sec:root_systems}
The initial data consists of an irreducible root system $R$, a choice of positive system $R_+ \subset R$, which determines a base of simple roots $\{\alpha_i\}_{i \in I}$ with $I = \{1,\dots,n\}$, and a root multiplicity $k : R \to \bbR$. The root system $R$ is allowed to be non-reduced, i.e., of type $BC_n$, for which it will be convenient to introduce the inmultiplicable roots $R^0 = \{\alpha \in R \mid 2\alpha \notin R\}$ and the root multiplicity $k^0$ on $R^0$ defined by $k^0(\alpha) = \tfrac{1}{2}k(\alpha/2) + k(\alpha)$, with the convention that $k(\alpha/2) = 0$ whenever $\alpha/2 \notin R$. We denote by $\eps : W \to \bbR^\times$ a linear character of the Weyl group $W$ of $R$, which governs the transformation behavior of the $\eps$-symmetric Heckman--Opdam polynomials.

Let $P$ be the weight lattice of $R$. The Heckman--Opdam polynomials are constructed as elements of the group algebra $\bbC[P]$ of $P$, viewed as exponential polynomials, which we will equip with an inner product. We use the ambient Euclidean space $(\fraka, (\cdot, \cdot))$ of $R$ to define the torus $T = i\fraka / 2\pi i \bbZ R^\vee$, where $R^\vee = \{\alpha^\vee := 2\alpha/(\alpha,\alpha) \mid \alpha \in R\}$ is the dual root system. The quotient map $\exp : i\fraka \to T$ has a multi-valued inverse $\log : T \to i\fraka$, which allows us to view the elements of $\bbC[P]$ as functions on $T$ by identifying $e^\lambda$ with $t^\lambda : t \mapsto e^{(\lambda, \log(t))}$. We denote by $M(k)$ the space $\bbC[P]$ together with the inner product\vspace{-0.1cm}
\[
    (f, g)_k = \int_T f(t) \overline{g(t)} \prod_{\alpha \in R} (1 - t^\alpha)^{k(\alpha)} dt,
\]
where $dt$ denotes the normalized Haar measure on $T$.

The space $M(k)$ carries a natural action of $W$ given transposition, i.e., $\phi^w(t) := \phi(w^{-1}t)$. For a character $\eps$, we denote the $\eps$-isotypical component of $M(k)$ by $M^{(\eps)}(k)$. In case of the trivial character, we simply write $M^W(k)$.


\subsection{Orthogonal polynomials}\label{sec:orthogonal_polynomials}
The non-symmetric Heckman--Opdam polynomials $E_\lambda(k)$ are defined by applying the Gram--Schmidt process to the basis $(e^\lambda)_{\lambda \in P}$ of $M(k)$ with respect to $(\cdot, \cdot)_k$ and a partial ordering $\leq$ of $P$. Let $P_+ \subset P$ be the cone of dominant weights and denote by $\lambda_+$ the unique element of $P_+$ in the Weyl orbit of $\lambda \in P$. For $\lambda \in P_+$, the group $W$ decomposes as $W^\lambda W_\lambda$, where $W_\lambda$ is the stabilizer of $\lambda$ in $W$ and $W^\lambda$ is the set of shortest coset representatives. Each $\lambda \in P$ can be written uniquely as $\overline{v}(\lambda) \lambda_+$ with $\overline{v}(\lambda) \in W^{\lambda_+}$. The partial ordering on $P_+$ is given by $\mu \leq \lambda \iff \lambda - \mu \in \bbZ_{\geq 0}R_+$. Writing $\leq_W$ for the Bruhat ordering of $W$, we extend the partial ordering $\leq$ to $P$ by
\[
    \lambda \leq \mu \iff
    \begin{cases}
        \lambda_+ \leq \mu_+ & \text{if $\lambda_+ \neq \mu_+$}; \\
        \overline{v}(\lambda) \leq_W \overline{v}(\mu) & \text{otherwise}.
    \end{cases}
\]

The polynomial $E_\lambda(k)$ is the unique element of $M(k)$ satisfying
\begin{enumerate}[(i)]
    \item $E_\lambda(k) = e^\lambda + \text{lower-order terms}$;
    \item $(E_\lambda(k), e^\mu)_k = 0 \text{ for all } \mu < \lambda$.
\end{enumerate}
For comparable weights $\lambda$ and $\mu$, it is a direct consequence of the definition that $E_\lambda(k)$ and $E_\mu(k)$ are orthogonal; however, for incomparable weights this is not immediate and should be proven, see e.g.~\cite[Cor.2.11]{Opd95}. Similarly, the $\eps$-symmetric Heckman--Opdam polynomials $P^{(\eps)}_\lambda(k)$ are defined by applying the Gram--Schmidt process to the basis $\big(m^{(\eps)}_\lambda := \sum_{w \in W_\lambda} \eps(w)e^{w\lambda}\big)_{\lambda \in P_+ \thinspace : \thinspace \eps(W_\lambda) = \{1\}}$ of $M^{(\eps)}(k)$, so that $P^{(\eps)}_\lambda(k)$ is the unique element of $M^{(\eps)}(k)$ satisfying
\begin{enumerate}[(i)]
    \item $P^{(\eps)}_\lambda(k) = m^{(\eps)}_\lambda + \text{lower-order terms}$;
    \item $(P^{(\eps)}_\lambda(k), m^{(\eps)}_\mu)_k = 0 \text{ for all } \mu < \lambda$.
\end{enumerate}
The condition $\eps(W_\lambda) = \{1\}$ is there to ensure that $m^{(\eps)}_\lambda \neq 0$, so that $P^{(\eps)}_\lambda(k) \neq 0$ as well. For\vspace{-0.1cm} the trivial character, the condition is trivially true and in this case we refer to $P_\lambda(k) := P_\lambda^{(\eps)}(k)$ simply as the (symmetric) Heckman--Opdam polynomials.

The non-symmetric and $\eps$-symmetric Heckman--Opdam polynomials are related by symmetrization with respect to the character $\eps$.

\begin{lemma}\label{lem:sym_prop}
    Let $\lambda \in P_+$ such that $\eps(W_\lambda) = \{1\}$. Then
    \[
        P^{(\eps)}_\lambda(k) = \frac{1}{|W_\lambda|} \sum_{w \in W} \eps(w) E^w_\lambda(k).
    \]
\end{lemma}

\begin{proof}
    This follows from the proof of \cite[Thm.2.12(1)]{Opd95}.
\end{proof}


\subsection{Dunkl--Cherednik operators}\label{sec:dunkl_cherednik_operators}
The polynomials $E_\lambda(k)$ are eigenfunctions of the Dunkl--Cherednik operators
\[
    T_\xi(k) = \partial_\xi + \sum_{\alpha \in R_+} k(\alpha) (\xi, \alpha)\frac{1 - r_\alpha}{1 - e^{-\alpha}} - (\xi, \rho_k), \qquad \xi \in \frakh := \fraka \otimes \bbC,
\]
with $\rho_k = \sum_{\alpha \in R_+} k(\alpha) \alpha$. More specifically,
\[
    T_\xi(k) E_\lambda(k) = (\xi, r_k(\lambda)) E_\lambda(k), \qquad \xi \in \frakh,
\]
where $r_k(\lambda) = \lambda + \frac{1}{2}\sum_{\alpha \in R_+} k(\alpha)\epsilon((\lambda, \alpha^\vee)) \alpha$. Here $\epsilon(x) = 1$ if $x > 0$ and $\epsilon(x) = -1$ otherwise.

Let $S(\frakh)$ be the symmetric algebra on $\frakh$, which is generated by the elements $\xi \in \frakh$. It is a well-known result that the Dunkl--Cherednik operators are mutually commutative, from which we see that the assignment $\xi \in \frakh \mapsto T_\xi(k) \in \End(M(k))$ extends to an algebra homomorphism $p \in S(\frakh) \mapsto T_p(k) \in \End(M(k))$. The algebra $S(\frakh)$ naturally comes with an action of $W$. The subalgebra of $W$-invariant elements is denoted by $S(\frakh)^W$. The spaces $M^{(\eps)}(k)$ are stable under the action of $T_p(k)$ whenever $p \in S(\frakh)^W$, and the $\eps$-symmetric polynomials $P^{(\eps)}_\lambda(k)$ are eigenfunctions of $T_p(k)$, viz.,
\[
    T_p(k) P^{(\eps)}_\lambda(k) = p(r_k(\lambda)) P^{(\eps)}_\lambda(k), \qquad p \in S(\frakh)^W.
\]
Alternatively, if $\lambda \in P_+$, then $r_k(\lambda)$ can be written as $w_\lambda(\lambda + \rho_k)$, where $w_\lambda$ denotes the longest element of $W_\lambda$. Thus, the eigenvalue $p(r_k(\lambda))$ equals $p(\lambda + \rho_k)$ for all $p \in S(\frakh)^W$. Furthermore, we have $r_k(\mu) = \overline{v}(\mu)r_k(\lambda) = \overline{v}(\mu)w_\lambda(\lambda + \rho_k)$ for all $\mu \in W\lambda$.

For any differential-reflection $T(k)$ operator acting on $M^W(k)$, there exists a unique differential operator $\Rad(T(k))$, the so-called radial part, which is obtained by replacing the reflections with the identity, so that it agrees with $T(k)$ on $M^W(k)$. Taking the radial parts of $T_p(k)$, we recover the system of hypergeometric differential equations from \cite[Def.2.13]{HeOp87}:
\begin{equation}\label{eq:syst_hypergeometric_diff_eq}
    \Rad(T_p(k)) P_\lambda(k) = p(\lambda + \rho_k) P_\lambda(k), \qquad p \in S(\frakh)^W.
\end{equation}


\subsection{Graded Hecke algebra}\label{sec:graded_hecke_algebra}
The Dunkl--Cherednik operators give rise to a representation of the graded Hecke algebra. Let us recall the definition of this algebra.

\begin{definition}\label{def:graded_Hecke_algebra}
    The graded Hecke algebra $\calH(k^0)$ is the vector space $S(\frakh) \otimes \bbC[W]$ equipped with the unique multiplication determined by
\begin{enumerate}[(i)]
    \item $S(\frakh) \to \calH(k^0), \thinspace p \mapsto p \otimes e$ and $\bbC[W] \to \calH(k^0), \thinspace w \mapsto 1 \otimes w$ are algebra homomorphisms;
    \item $(p \otimes e)(1 \otimes w) = p \otimes w$ for all $p \in S(\frakh)$ and $w \in W$;
    \item $(1 \otimes r_i) (\xi \otimes e) = r_i(\xi) \otimes r_i - k^0(\alpha_i) (\xi,\alpha_i)$ for all $i \in I$ and $\xi \in \frakh$.
\end{enumerate}
Note we have written $r_i$ for $i \in I$ as a shorthand for $r_{\alpha_i}$.
\end{definition}

\noindent The maps $p \otimes e \mapsto T_p(k) \in \End(M(k))$ and $1 \otimes w \mapsto w \in \End(M(k))$ extend to a faithful representation of $\calH(k^0)$. The center of $\calH(k^0)$ coincides with $S(\frakh)^W$, so that the module
\[
    M_\lambda(k) = \big\{\phi \in M(k) \mid T_p(k)\phi = p(r_k(\lambda))\phi, \quad p \in S(\frakh)^W\big\}, \qquad \lambda \in P_+,
\]
defines an $\calH(k^0)$-module with central character $r_k(\lambda)$. As a vector space, $M_\lambda(k)$ is isomorphic to $\bbC[W/W_\lambda]$, and a basis is given by $(E_\mu(k))_{\mu \in W\lambda}$. For $\mu \in W\lambda$, the weight space
\[
    M_\lambda^\mu(k) = \big\{\phi \in M(k) \mid T_\xi(k)\phi = (\xi, r_k(\mu))\phi, \quad \xi \in \frakh\big\}
\]
is one-dimensional and spanned by $E_\mu(k)$.

The module $M_\lambda(k)$ may also contains an $\eps$-isotypical component for some character $\eps$, i.e.,
\[
    M_\lambda^{(\eps)}(k) = \big\{\phi \in M^{(\eps)}(k) \mid T_p(k)\phi = p(r_k(\lambda))\phi, \quad p \in S(\frakh)^W\big\},
\]
which is at most one-dimensional. Let us define the root multiplicity
\[
    l(\alpha) = 
    \begin{cases}
        1 & \text{if } \eps(r_\alpha) = -1; \\
        0 & \text{otherwise}.
    \end{cases}
\]
It can be shown that
\[
    \dim M^{(\eps)}_\lambda(k) = 1 \iff \eps(W_\lambda) = \{1\} \iff \lambda \in \rho_l + P_+,
\]
in which case $M_\lambda^{(\eps)}(k)$ is spanned by $P^{(\eps)}_\lambda(k)$, cf.~the proof of Lemma \ref{lem:dim_equivalence_q}.

The polynomial $P^{(\eps)}_\lambda(k)$ is a certain linear combination of the elements of the basis $(E_\mu(k))_{\mu \in W\lambda}$ of $M_\lambda(k)$, cf.~Lemma \ref{lem:sym_prop}, and the coefficients can be described in terms of the $\bfc$-functions
\[
    \bfc_k^\pm(w)(\cdot) = \prod_{\alpha \in R^0_+ \cap w^{-1}R^0_\pm}\left(1 + \frac{k^0(\alpha)}{(\cdot, \alpha^\vee)}\right), \qquad w \in W.
\]

\begin{assumption}
    The $\bfc$-functions will often be evaluated at $\lambda + \rho_k$ or at $r_k(\mu)$ with $\lambda \in P_+$ and $\mu \in W\lambda$. In order to ensure that this is well-defined, we impose the running assumption on the multiplicity $k$ which states that $k^0(\alpha) > 0$ for all $\alpha \in R^0$. This guarantees that $\lambda + \rho_k$ and $r_k(\mu)$ are always regular, so that $(\lambda + \rho_k, \alpha^\vee)$ and $(r_k(\mu), \alpha^\vee)$ remain non-zero for all $\alpha \in R^0$.
\end{assumption}

\begin{lemma}\label{lem:eHO_in_nHO_basis}
    Let $\lambda \in \rho_l + P_+$. Then
    \[
        P^{(\eps)}_\lambda(k) = \frac{1}{|W_\lambda|}\sum_{\mu \in W\lambda} \eps(\overline{v}(\mu))\bfc_{-\eps k}^+(\overline{v}(\mu))(r_k(\lambda))E_\mu(k),
    \]
    where $\eps k$ denotes the root multiplicity $\alpha \mapsto \eps(r_\alpha)k(\alpha)$.
\end{lemma}

\begin{proof}
    For the trivial character, the statement follows directly from \cite[Thm.4.1,Lem.5.1]{Opd95}. Moreover, the arguments from \cite{Opd95} can also be used to give a proof for the remaining cases.
\end{proof}


\subsection{Symmetric shift operators}\label{sec:sym_shift_operators}
The symmetric shift operators relate the symmetric Heckman--Opdam polynomials with different multiplicities. We recall the construction of the $\eps$-forward shift operator from \cite{Hec91}, following \cite{Hec97,Opd00}. The cornerstone of this construction is the shift principle. 

Let $\Delta_\eps$ be the Weyl denominator relative to the character $\eps$, i.e.,
\[
    \Delta_\eps = \prod_{\alpha \in R^0_+, \thinspace l(\alpha) = 1}(e^{\alpha/2} - e^{-\alpha/2}).
\]

\begin{proposition}[Shift principle]\label{prop:shift_prin}
    Let $\lambda \in \rho_l + P_+$. Then
    \[
        P_{\lambda - \rho_l}(k + l) = \Delta_\eps^{-1} P^{(\eps)}_\lambda(k).
    \]
\end{proposition}

\begin{proof}
    We refer the reader to \cite[Thm.5.8]{Opd00} for the case of the sign character, and the proof carries over mutatis mutandis for the remaining cases.
\end{proof}

Writing $\pi^{(\eps)}_+(k) = \prod_{\alpha \in R^0_+, \thinspace l(\alpha) = 1}\left(T_{\alpha^\vee}(k) + k^0(\alpha)\right)$, the $\eps$-forward shift operator $G^{(\eps)}_+(k)$ is defined to be the differential operator
\[
    G^{(\eps)}_+(k) = \Rad\big(\Delta_\eps^{-1}T_{\pi^{(\eps)}_+(k)}(k)\big).
\]
A backward variant can be defined similarly, see \cite[Def.5.9]{Opd00}.

The following theorem can be found in \cite[Thm.5.10]{Opd00}; however, there it is only stated for the sign character. We revisit its proof to illustrate that it applies to any character.

\begin{theorem}\label{thm:sym_shift_factor}
    The $\eps$-forward shift operator $G^{(\eps)}_+(k)$ shifts the multiplicity of the symmetric Heck-man--Opdam polynomials by $l$:
    \[
        G^{(\eps)}_+(k)P_\lambda(k) = h^{(\eps)}_+(\lambda, k) P_{\lambda - \rho_l}(k + l), \qquad \lambda \in P_+,
    \]
    with $h^{(\eps)}_+(\lambda, k) = \prod_{\alpha \in R^0_+, \thinspace l(\alpha) = 1}\big((\lambda + \rho_k, \alpha^\vee) - k^0(\alpha)\big)$.
\end{theorem}

\begin{proof}
    From Definition \ref{def:graded_Hecke_algebra}(iii), it can be shown that
    \[
        r_i T_{\pi^{(\eps)}_+}(k)\big|_{M^W(k)} = \eps(r_i)T_{\pi^{(\eps)}_+}(k)r_i\big|_{M^W(k)},
    \]
    cf.~the proof of \cite[Lem.5.8(a)]{Opd00}. It follows that $T_{\pi^{(\eps)}_+}(k)P_\lambda(k) = cP^{(\eps)}_\lambda(k)$ for some constant $c$.\vspace{-0.25cm} This shows that $G^{(\eps)}_+(k)P_\lambda(k) = 0$ when $\lambda - \rho_l \notin P_+$ since then $P^{(\eps)}_\lambda(k) = 0$. This agrees with the theorem because $P_{\lambda - \rho_l}(k + l) = 0$ in this case. Suppose that $\lambda \in \rho_l + P_+$. Using the triangularity of the Dunkl--Cherednik operators, we arrive at
    \[
        T_{\pi^{(\eps)}_+}(k)e^{w_0\lambda} = \eps(w_0)h^{(\eps)}_+(\lambda, k)e^{w_0\lambda} + \text{lower-order terms},
    \]
    so that $c = h^{(\eps)}_+(\lambda, k)$ after comparing the coefficients of $e^{w_0\lambda}$. Therefore,
    \[
        G^{(\eps)}_+(k)P_\lambda(k) = h^{(\eps)}_+(\lambda, k)\Delta_\eps^{-1}P^{(\eps)}_\lambda(k) = h^{(\eps)}_+(\lambda, k)P_{\lambda - \rho_l}(k + l),
    \]
    using shift principle (Proposition \ref{prop:shift_prin}) for the second equality.
\end{proof}

A distinctive feature of the symmetric shift operators is the transmutation property with the differential operators corresponding to the system of hypergeometric differential equations \eqref{eq:syst_hypergeometric_diff_eq}. The transmutation property of $G^{(\eps)}_+(k)$ is an immediate consequence of Theorem~\ref{thm:sym_shift_factor}.

\begin{corollary}\label{cor:sym_transmutation_prop}
    The operator $G^{(\eps)}_+(k)$ satisfies the transmutation property
    \[
        G^{(\eps)}_+(k) \Rad(T_p(k)) = \Rad(T_p(k + l)) G^{(\eps)}_+(k), \qquad p \in S(\frakh)^W,
    \]
    viewed as operators from $M(k)$ to $M(k \pm l)$.
\end{corollary}

\begin{proof}
    We recall from \cite[Prop.1.3.7]{HeSc94} that differential operators acting on $M(k)$ are uniquely determined by their action on the invariant subspace $M^W(k)$. Therefore, it suffices to verify that $G^{(\eps)}_+(k) \Rad(T_p(k))$ and $\Rad(T_p(k + l)) G^{(\eps)}_+(k)$ agree when applied to the basis $(P_\lambda(k))_{\lambda \in P_+}$ of $M^W(k)$. Applying both operators to $P_\lambda(k)$ yields multiples of $P_{\lambda - \rho_l}(k + l)$ with coefficients
    \[
        h^{(\eps)}_+(\lambda,k)p(-\lambda - \rho_{k'}), \quad h^{(\eps)}_+(\lambda,k)p(-(\lambda - \rho_l) - \rho_{k + l})
    \]
    coming from Theorem \ref{thm:sym_shift_factor}. The equality of these coefficients follows from noticing that $\lambda - \rho_l + \rho_{k + l} = \lambda + \rho_k$, which completes the proof of the corollary.
\end{proof}


\subsection{KZ-equations, Cherednik--Matsuo correspondence, and the KZ-shift operator}\label{sec:kz_equations}
The Cherednik--Matsuo correspondence describes a relationship between the solutions of the spectral problem of Dunkl--Cherednik operators and solutions of the KZ-equations. In particular, it states that $\eps$-symmetric Heckman--Opdam polynomials correspond to polynomial solutions of the KZ-equations. Matsuo \cite{Mat92} proved a correspondence between the solutions of the system of hypergeometric differential equations \eqref{eq:syst_hypergeometric_diff_eq} and the solutions of the KZ-equations
\[
    \text{KZ}_\lambda(k) = \big\{\Phi \in (\bbC[P] \otimes \bbC[W])^W \mid \nabla_\xi(\lambda, k)\Phi = 0, \quad \xi \in \frakh\big\}, \qquad \lambda \in P,
\]
where $\nabla(\lambda, k)$ is the KZ-connection. The covariant differentiation is given by
\[
    \begin{aligned}
        \nabla_\xi(\lambda, k)(\phi \otimes w) &= \partial_\xi \phi \otimes w \\
        &+ \frac{1}{2}\sum_{\alpha \in R_+} k(\alpha)(\xi, \alpha) \left(\frac{1 + e^{-\alpha}}{1 - e^{-\alpha}} \phi \otimes (1 - r_\alpha) w - \phi \otimes r_\alpha \eps_\alpha w\right) - (\xi, wr_k(\lambda)) \phi \otimes w,
    \end{aligned}
\]
with
\[
    r_\alpha(\phi \otimes w) = \phi \otimes (r_\alpha w), \qquad
    \eps_\alpha(\phi \otimes w) =
    \begin{cases}
        \phi \otimes w & \text{if } w^{-1}(\alpha) \in R_+; \\
        -(\phi \otimes w) & \text{otherwise}.
    \end{cases}
\]

A generalization of this correspondence was studied by Cherednik in \cite{Che94}. Felder and Veselov~\cite{FeVe94} used this correspondence to construct a shift operator for the KZ-equations. For regular $\lambda \in P_+$ and the sign character, they considered the maps
\[
    \begin{aligned}
        \text{m}_\lambda(k) &: \text{KZ}_\lambda(k) \to M^W_\lambda(k), \thinspace \sum_{w \in W} \phi_w \otimes w \mapsto \sum_{w \in W} \phi_w; \\
        \text{ch}_\lambda(k) &: \text{KZ}_\lambda(k) \to M_{\lambda - \rho_l}^W(k + l), \thinspace \sum_{w \in W} \phi_w \otimes w \mapsto \Delta_\eps^{-1}\sum_{w \in W} \eps(w)\phi_w,
    \end{aligned}
\]
which are isomorphisms due to the Cherednik--Matsuo correspondence. In \cite[\S3]{FeVe94}, the authors invert the isomorphism $\text{m}_\lambda(k)$ explicitly\footnote{Only the rational case is under consideration by the authors and they assert that it can be achieved in the trigonometric case as well.} and use this inverse to define the KZ-shift operator as the following composition:
\[
    \text{m}_{\lambda - \rho}(k + l)^{-1} \circ \text{ch}_\lambda(k) : \text{KZ}_\lambda(k) \to \text{KZ}_{\lambda - \rho}(k + l).
\]

For $\lambda \in P_+$ and $\mu \in W\lambda$, there is an isomorphism
\[
    M_\lambda^\mu(k) \to \text{KZ}_\mu(k), \thinspace \phi \mapsto \sum_{w \in W} \phi^w \otimes w,
\]
see \cite[\S3]{Opd95}. Thus, the pullback of the KZ-shift operator provides a candidate shift operator for the non-symmetric Heckman--Opdam polynomials for fixed regular weights $\mu = \lambda$. The limitation of this operator is that it depends on the spectral parameter, the weight $\lambda$. This dependence on the weight precludes the operator from being a differential-reflection operator satisfying a certain transmutation property with the Dunkl--Cherednik operators.


\section{Shift operators for non-symmetric Heckman--Opdam polynomials}\label{sec:non_sym_shift_operators}\vspace{-0.1cm}
We begin by providing a detailed account of the non-symmetric shift operators for fixed weights.\vspace{-0.05cm}


\subsection{Non-symmetric shift operators for fixed weights}\label{sec:non_sym_shift_operators_1}
Taking inspiration from the construction of the KZ-shift operator from Section \ref{sec:kz_equations}, we consider the following sequence of maps:
\[
    M_\lambda^\mu(k) \overset{\text{(i)}}{\longrightarrow} M^{(\eps_\pm)}_\lambda(k) \overset{\text{(ii)}}{\longrightarrow} M^W_{\lambda\mp\rho_l}(k \pm l) \overset{\text{(iii)}}{\longrightarrow} M_{\lambda\mp\rho_l}^{\mu_{\eps,\pm}}(k \pm l),
\]
with the notation $\eps_+ = \eps$, $\eps_- = \triv$, and $\mu_{\eps,\pm} = \overline{v}(\mu)w_\lambda(\lambda\mp\rho_l)$, where:
\begin{enumerate}[(i)]
    \item $U_{\eps_\pm} : M_\lambda^\mu(k) \to M_\lambda^{(\eps_\pm)}(k)$ is the $\eps_\pm$-symmetrizer defined by $U_{\eps_\pm}(\phi) = \sum_{w \in W}\eps_\pm(w)\phi^w$;
    \item $\Delta_\eps^\mp : M_\lambda^{(\eps_\pm)}(k) \to M_{\lambda\mp\rho_l}^W(k \pm l)$ given by multiplication with $\Delta_\eps^{\mp 1}$, which is how the shift principle (Proposition \ref{prop:shift_prin}) manifests itself in this approach;
    \item $Q_{\mu_{\eps,\pm}}(k \pm l) : M_{\lambda\mp\rho_l}(k \pm l) \to M_{\lambda\mp\rho_l}^{\mu_{\eps,\pm}}(k \pm l)$ is a certain scalar multiple of the projection from $M_{\lambda\mp\rho_l}(k \pm l)$ onto the subspace $M_{\lambda\mp\rho_l}^{\mu_{\eps,\pm}}(k \pm l)$.
\end{enumerate}
Taking the composition of these maps yields the $\eps$-forward/backward non-symmetric shift operator for fixed weights $\lambda \in P_+$ and $\mu \in W\lambda$:
\[
    \calG^{(\eps)}_\pm(\mu, k) = Q_{\mu_{\eps,\pm}}(k \pm l) \circ \Delta_\eps^\mp \circ U_{\eps_\pm} : M_\lambda^\mu(k) \to M_{\lambda \mp \rho_l}^{\mu_{\eps,\pm}}(k \pm l).
\]
Surprisingly, this fixed-weight non-symmetric shift operator provides us with an ansatz for a genuine non-symmetric shift operator that is no longer dependent on the weights $\lambda$ nor $\mu$. Generically, the non-symmetric shift operators agree with one another; however, this may fail to hold for degenerate cases of $\mu$.

We proceed to examine the operator $Q_{\mu_{\eps,\pm}}(k \pm l)$ in detail. The following lemma is essential for its definition.

\begin{lemma}\label{lem:minimal_poly}
    Writing $\pi = \prod_{\alpha \in R^0_+}(\cdot,\alpha^\vee)$, there exists a polynomial $q \in S(\frakh) \otimes S(\frakh)$ such that
    \begin{equation}\label{eq:minimal_poly_prop}
        q(\xi, w\xi) = \pi(\xi)\delta_{e,w}, \qquad w \in W,
    \end{equation}
    for all regular $\xi \in \frakh$.
\end{lemma}

\begin{proof}
    The polynomial $q$ is constructed in \cite[Lem.3.9(1)]{Opd95}, which is based on \cite[Ch.4:Ex.70]{Var84}. Since the polynomial $q$ is integral to what follows, we expound on its construction. Let $H(\frakh) \subset S(\frakh)$ be the space of harmonic polynomials on $\frakh$ and recall that it is isomorphic to $\bbC[W]$ as a vector space. With respect to the canonical basis $(\delta_w)_{w \in W}$ of $\bbC[W]$ and a fixed regular $\xi \in \frakh$, a linear isomorphism is given by
    \[
        U(\xi) : H(\frakh) \to \bbC[W], \thinspace h \mapsto \sum_{w \in W} h^w(\xi)\delta_w.
    \]
    Defining $q(\xi, \cdot) = \pi(\xi)U(\xi)^{-1}(\delta_e) \in H(\frakh)$, for regular $\xi \in \frakh$, satisfies the desired property \eqref{eq:minimal_poly_prop}; however, a priori, it is a rational function in $\xi$.
    
    To show that $q$ is in fact polynomial, we take a basis $u_1,\dots,u_d$ of $H(\frakh)$, with $d = |W|$, consisting of homogeneous polynomials, so that the linear isomorphism $U(\xi)$ is represented by the matrix $\hat{U}(\xi) = (u_j^w(\xi))_{1\leq j\leq d, \thinspace w \in W}$. As noted in \cite[Ch.4:Ex.70]{Var84}, the determinant of $U(\xi)$ is a non-zero multiple of $\pi(\xi)^{d/2}$ and the minors of $\hat{U}(\xi)$ are divisible by $\pi(\xi)^{d/2 - 1}$. Therefore, $\pi(\xi)\hat{U}(\xi)^{-1}$ has polynomial entries in $\xi$, which implies that $q$ is polynomial in $\xi$. Explicitly, we find that
    \begin{equation}\label{eq:minimal_poly_exp}
        q = \sum_{j = 1}^d q_j\otimes u_j,
    \end{equation}
    where the polynomials $q_1(\xi), \dots, q_d(\xi)$ are the entries of the first column of $\pi(\xi)\hat{U}(\xi)^{-1}$.
\end{proof}

For $\mu \in P$, we define the operator $Q_\mu(k) = q(r_k(\mu),T(k))$, where $q$ is the polynomial of Lemma~\ref{lem:minimal_poly}. Explicitly, in terms of the expansion \eqref{eq:minimal_poly_exp}, the operator is given by
\[
    Q_\mu(k) = \sum_{j = 1}^d q_j(r_k(\mu))T_{u_j}(k).
\]
The action of the operator $Q_\mu(k)$ is given by the following lemma, cf.~\cite[Lem.3.9(2)]{Opd95}.

\begin{lemma}\label{lem:proj_operator}
    Let $\lambda \in P_+$ and $\mu \in W\lambda$. The operator $Q_\mu(k)$ defines a map from $M_\lambda(k)$ to $M_\lambda^\mu(k)$. In particular, we have
    \[
        Q_\mu(k)E_\nu(k) = \delta_{\mu,\nu}\pi(r_k(\mu))E_\mu(k).
    \]
\end{lemma}

\begin{proof}
    By our running assumption on $k$, we have that $r_k(\mu)$ is regular. Therefore,
    \[
        Q_\mu(k)E_\nu(k) = q(r_k(\mu),r_k(\nu))E_\nu(k) = \delta_{\mu,\nu}\pi(r_k(\mu))E_\mu(k),
    \]
    using that $r_k(\nu) = \overline{v}(\nu)\overline{v}(\mu)^{-1}r_k(\mu)$.
\end{proof}

\begin{example}\label{exa:minimal_poly}\leavevmode
    \begin{enumerate}[(i)]
        \item For the root system $R = \{\pm\epsilon_1,\pm2\epsilon_1\}$ of type $BC_1$, the polynomial $q$ is given by $\frac{1}{4}\big(\epsilon_1^\vee \otimes 1 + 1 \otimes \epsilon_1^\vee\big)$, so that $Q_\mu(k) = \frac{1}{4}\big((r_k(\mu),\epsilon_1^\vee) + T_{\epsilon_1^\vee}(k)\big)$.
        \item Let $R = \{\epsilon_i - \epsilon_j \mid i, j = 1,2,3, \thinspace i \neq j\}$ be the root system of type $A_2$. Writing $\alpha_1 = \epsilon_1 - \epsilon_2$ and $\alpha_2 = \epsilon_2 - \epsilon_3$ for the standard choice of simple roots, a homogeneous basis of $H(\frakh)$ is given by
        \[
            \begin{aligned}
                &u_1 = 1, \quad u_2 = \alpha_1, \quad u_3 = \alpha_2, \quad u_4 = 2\alpha_1(\alpha_1 + 2\alpha_2); \\
                &\quad u_5 = 2\alpha_2(2\alpha_1 + \alpha_2), \qquad u_6 = 3\alpha_1\alpha_2(\alpha_1 + \alpha_2).
            \end{aligned}
        \]
        In terms of this basis, the polynomial $q$ takes the form $\frac{1}{18}\sum_{j = 1}^6 u_{7 - j} \otimes u_j$.
    \end{enumerate}
\end{example}

\begin{remark}
    There is an alternative choice for the polynomial $q$ from Lemma \ref{lem:minimal_poly}, given by
    \[
        q_{\xi'} = \prod_{w \in W, \thinspace w \neq e}\big(\xi' \otimes 1 - 1 \otimes w^{-1}\xi'\big),
    \]
    which depends on an auxiliary vector $\xi' \in \frakh$ satisfying $q_{\xi'}(\xi,\xi) \neq 0$. The polynomial $q_{\xi'}$ was used in \cite[\S3]{FeVe94} and \cite[Rmk.7.4]{Opd00} because it satisfies a property similar to \eqref{eq:minimal_poly_prop}, i.e., $q_{\xi'}(\xi, w\xi) = q_{\xi'}(\xi,\xi)\delta_{e,w}$ for all $w \in W$. Note that $q_{\xi'}$ is easier to define; however, it comes with two disadvantages. First, it depends on a vector $\xi'$ that has to satisfy a regularity condition with respect to $\xi$. Second, the degree of $q_{\xi'}$ equals $|W|$, while $q$ is of a lower degree, i.e., $|R_0^+|$, see Section \ref{sec:non_sym_shift_operators_prop} for the significance of the degree of $q$.
\end{remark}

Next, we compute the action of the operator $\calG^{(\eps)}_\pm(\mu, k)$ on the polynomials $E_\mu(k)$, i.e., we deter\vspace{-0.1cm}-mine the constant $H^{(\eps)}_\pm(\mu, k)$ such that $\calG^{(\eps)}_\pm(\mu, k)E_\mu(k) = H^{(\eps)}_\pm(\mu, k) E_{\mu_{\eps,\pm}}(k \pm l)$. It will be favorable to rewrite the arguments of all $\bfc$-functions in terms of $\lambda + \rho_k$ instead of $r_k(\lambda)$, which will be part of the shift factor $H^{(\eps)}_\pm(\mu, k)$. The first instance of this is Lemma \ref{lem:eHO_in_nHO_basis}.

\begin{lemma}\label{lem:poincare_poly}
    For $\lambda \in P_+$, we have $\bfc_k^-(w_\lambda)(\lambda + \rho_k)= |W_\lambda|$.
\end{lemma}

\begin{proof}
    Writing $R^0_\lambda = \{\alpha \in R^0 \mid (\lambda,\alpha) = 0\}$ and $R^0_{\lambda,+} = R^0_\lambda \cap R^0_+$, we note that $R^0_+ \cap w_\lambda^{-1}R^0_- = R^0_{\lambda,+}$. It follows that\vspace{-0.1cm}
    \[
        \bfc_k^-(w_\lambda)(\lambda + \rho_k) = \prod_{\alpha \in R^0_{\lambda,+}}\left(1 + \frac{k^0(\alpha)}{(\rho_k, \alpha^\vee)}\right) = |W_\lambda|,
    \]
    where the second equality is precisely the identity from \cite[(4.4)]{Opd95}.
\end{proof}

\begin{lemma}\label{lem:eHO_in_nHO_basis_v2}
    Let $\lambda \in \rho_l + P_+$. Then
    \[
        P^{(\eps)}_\lambda(k) = \sum_{\mu \in W\lambda} \eps(\overline{v}(\mu))\bfc_{-\eps k}^+(\overline{v}(\mu)w_\lambda)(\lambda + \rho_k)E_\mu(k).
    \]
\end{lemma}

\begin{proof}
    Using the decomposition
    \[
        w_\lambda(R^0_+ \cap \overline{v}(\mu)^{-1}R^0_-) = (R^0_+ \cap (\overline{v}(\mu)w_\lambda)^{-1}R^0_-)  \sqcup R^0_{\lambda,-},
    \]
    see e.g.~\cite[p.98]{Opd95}, we find that
    \[
        \bfc_{-\eps k}^+(\overline{v}(\mu))(r_k(\lambda)) = \bfc_{-\eps k}^+(\overline{v}(\mu)w_\lambda)(\lambda + \rho_k)\bfc^-_k(w_\lambda)(\lambda + \rho_k).
    \]
    Here we are allowed to replace the multiplicity $-\eps k$ with $k$ since $R^0_{\lambda,-} = -(R^0_+ \cap w_\lambda^{-1}R^0_-)$ and $\eps k \equiv k$ on $R^0_\lambda$. The lemma now follows from Lemma \ref{lem:eHO_in_nHO_basis} and Lemma \ref{lem:poincare_poly}.
\end{proof}

From Lemma \ref{lem:proj_operator}, we see that the action of the operator $Q_\mu(k)$ on $P^{(\eps)}_\lambda(k)$ is captured precisely by the coefficients from Lemma \ref{lem:eHO_in_nHO_basis_v2}. This brings us to the following corollary.

\begin{corollary}\label{cor:proj_operator_eHO}
    Let $\lambda \in \rho_l + P_+$ and $\mu \in W\lambda$. Then
    \[
        Q_\mu(k)P^{(\eps)}_\lambda(k) = \eps(\overline{v}(\mu))\pi(r_k(\mu))\bfc_{-\eps k}^+(\overline{v}(\mu)w_\lambda)(\lambda + \rho_k)E_\mu(k).
    \]
\end{corollary}

In case that the weight $\mu$ is dominant, i.e., $\mu = \lambda$, we are now able to calculate the action of $\calG^{(\eps)}_\pm(\lambda,k)$ on $E_\lambda(k)$ using that $P^{(\eps)}_\lambda(k) = |W_\lambda|^{-1}U_\eps E_\lambda(k)$. For general weights $\mu$, however, we first have to relate $U_\eps E_\mu(k)$ and $P^{(\eps)}_\lambda(k)$ if we intent to use the shift principle (Proposition \ref{prop:shift_prin}).

\begin{lemma}\label{lem:relating_U_nHO_and_eHO}
    Let $\lambda \in \rho_l + P_+$ and $\mu \in W\lambda$. Then
    \[
        U_\eps E_\mu(k) = \eps(\overline{v}(\mu))\bfc_{\eps k}^-(\overline{v}(\mu)w_\lambda)(\lambda + \rho_k)P^{(\eps)}_\lambda(k).
    \]
\end{lemma}

\begin{proof}
    Adapting the proof of \cite[Lem.4.5]{Sch23} to our setting results in
    \[
        U_\eps E_\mu(k) = \eps(\overline{v}(\mu))\bfc_{\eps k}^-(\overline{v}(\mu))(r_k(\lambda))U_\eps E_\lambda(k).
    \]
    There is a second decomposition given in \cite[p.98]{Opd95}, i.e.,
    \[
        w_\lambda^{-1}(R^0_+ \cap \overline{v}(\mu)^{-1}R^0_-) \sqcup R^0_{\lambda,+} = R^0_+ \cap (\overline{v}(\mu)w_\lambda)^{-1}R^0_-,
    \]
    which we use to arrive at
    \[
        \bfc_{\eps k}^-(\overline{v}(\mu))(r_k(\lambda)) = \bfc_{\eps k}^-(\overline{v}(\mu)w_\lambda)(\lambda + \rho_k)\bfc^-_k(w_\lambda)(\lambda + \rho_k)^{-1},
    \]
    cf.~the proof of Lemma \ref{lem:eHO_in_nHO_basis_v2}. Substituting this identity into the expression for $U_\eps E_\mu(k)$ together with Lemma \ref{lem:poincare_poly} and $P^{(\eps)}_\lambda(k) = |W_\lambda|^{-1}U_\eps E_\lambda(k)$, see Lemma \ref{lem:sym_prop}, gives the lemma.
\end{proof}

In recasting the $\bfc$-functions, we relied on the assumption $\lambda \in \rho_l + P_+$. When this is not the case, i.e., $\lambda - \rho_l \notin P_+$, the action of $\calG^{(\eps)}_+(\mu,k)$ on $E_\mu(k)$ is trivially $0$ as $U_\eps E_\mu(k) = 0$. The next lemma guarantees that the $\bfc$-functions in the coefficients take this vanishing into account.

\begin{lemma}[Vanishing Lemma 1]\label{lem:vanishing_1}
    Let $\lambda \in P_+$ and $\mu \in W\lambda$. Assuming that $\lambda - \rho_l \notin P_+$, we have
    \[
        \bfc_{\eps k}^-(\overline{v}(\mu)w_\lambda)(\lambda + \rho_k) = 0.
    \]
\end{lemma}

\begin{proof}
    There exists an $i \in I$ such that the root $\alpha_i$ is an element of $R^0_{\lambda,+}$ and $\eps(r_i) = -1$ by virtue of $\eps(W_\lambda) \neq \{1\}$. Let us recall a well-known characterization of $W^\lambda$:
    \begin{equation}\label{eq:shortest_coset_char}
        w \in W^\lambda \iff w(R^0_{\lambda,+}) \subset R^0_+.
    \end{equation}
    This characterization is immediate from the definition $W^\lambda = \{w \in W \mid \ell(wr_i) > \ell(w) \text{ for all } i \in I \linebreak\text{satisfying } \alpha_i \in R^0_{\lambda,+}\}$ and the fact that $\ell(wr_i) = \ell(w) \pm 1$ if and only if $w(\alpha_i) \in R^0_\pm$. Since $\overline{v}(\mu) \in W^\lambda$ and $w_\lambda(\alpha_i) \in R^0_{\lambda,-}$, we find that $(\overline{v}(\mu)w_\lambda)(\alpha_i) \in R^0_-$ by \eqref{eq:shortest_coset_char}, or equivalently, $\alpha_i \in R^0_+ \cap (\overline{v}(\mu)w_\lambda)^{-1}R^0_-$. It follows that $(\lambda + \rho_k, \alpha_i^\vee) + \eps(r_i) k^0(\alpha_i) = (\rho_k, \alpha_i^\vee) - k^0(\alpha_i) = 0$. This implies that $\bfc_{\eps k}^-(\overline{v}(\mu)w_\lambda)(\lambda + \rho_k) = 0$ since we have shown that one of its factors vanishes.
\end{proof}

Finally, we combine the preceding lemmata to calculate the shift factor $H^{(\eps)}_\pm(\mu, k)$.

\begin{proposition}\label{prop:fw_shift_factor}
    Let $\lambda \in P_+$ and $\mu \in W\lambda$. Then
    \[
        \calG^{(\eps)}_\pm(\mu, k) E_\mu(k) = H^{(\eps)}_\pm(\mu, k) E_{\mu_{\eps,\pm}}(k \pm l),
    \]
    with
    \[
        \begin{aligned}
            H^{(\eps)}_\pm(\mu, k) &= \eps_\pm(\overline{v}(\mu))\eps_\mp(\overline{v}(\mu_{\eps,\pm}))\pi(r_{k \pm l}(\mu_{\eps,\pm}))\bfc_{\eps_\pm k}^-(\overline{v}(\mu)w_\lambda)(\lambda + \rho_k) \\
            &\qquad\times\bfc_{-\eps_\mp(k \pm l)}^+(\overline{v}(\mu_{\eps,\pm})w_{\lambda\mp\rho_l})(\lambda + \rho_k).
        \end{aligned}
    \]
\end{proposition}

\begin{proof}
    First, we suppose that $\lambda \mp \rho_l \in P_+$. By the shift principle (Proposition \ref{prop:shift_prin}) and Lemma~\ref{lem:relating_U_nHO_and_eHO}, we have
    \[
        (\Delta^\mp_\eps\circ U_{\eps_\pm})E_\mu(k) = \eps_\pm(\overline{v}(\mu))\bfc^-_{\eps_\pm k}(\overline{v}(\mu)w_\lambda)(\lambda + \rho_k)P^{(\eps_\mp)}_{\lambda\mp\rho_l}(k\mp l).
    \]
    The proposition follows from this in conjunction with Corollary \ref{cor:proj_operator_eHO} applied to $Q_{\mu_{\eps,\pm}}(k \pm l)$, i.e.,
    \[
        Q_{\mu_{\eps,\pm}}(k \pm l) P^{(\eps_\mp)}_{\lambda\mp\rho_l}(k\mp l) = \eps_\mp(\overline{v}(\mu_{\eps,\pm}))\bfc_{-\eps_\mp(k \pm l)}^+(\overline{v}(\mu_{\eps,\pm})w_{\lambda\mp\rho_l})(\lambda + \rho_k)\pi(r_{k \pm l}(\mu_{\eps,\pm}))E_{\mu_{\eps,\pm}}(k \pm l).
    \]
    
    Second, the only thing that can occur is that $\lambda - \rho_l \notin P_+$. It follows that $U_\eps E_\mu(k) = 0$. Therefore, $\calG^{(\eps)}_+(\mu, k)E_\mu(k) = 0$, so that we have to prove that $H^{(\eps)}_+(\mu, k) = 0$. In Lemma \ref{lem:vanishing_1}, we have shown that one of its factors vanishes, which completes the proof of the proposition.
\end{proof}


\subsection{Genuine non-symmetric shift operators}\label{sec:non_sym_shift_operators_2}
We proceed by rewriting the fixed-weight non-symmetric shift operator $\calG^{(\eps)}_\pm(\mu, k)$ as follows:
\begin{equation}\label{eq:fw_shift_operator_rewritten}
    \calG^{(\eps)}_\pm(\mu, k) = \sum_{w \in W} \eps_\pm(w) Q_{\mu_{\eps,\pm}}(k \pm l)\Delta_\eps^\mp w,
\end{equation}
and observing that the $\mu$-dependence that we wish to eliminate resides entirely within the operator $Q_{\mu_{\eps,\pm}}(k \pm l)$. There is a certain unutilized freedom available in \eqref{eq:fw_shift_operator_rewritten} for the operator $Q_{\mu_{\eps,\pm}}(k \pm l)$ since we are allowed to vary it with $w \in W$. More specifically, we seek to replace the operator $Q_{\mu_{\eps,\pm}}(k \pm l)$ with operators $\calQ^{(\eps)}_\pm(w, k)$, with $w \in W$, satisfying
\begin{equation}\label{eq:proj_operator_replacement}
    \calQ^{(\eps)}_\pm(w, k)\Delta_\eps^\mp w\big|_{M_\lambda^\mu(k)} = Q_{\mu_{\eps,\pm}}(k \pm l)\Delta_\eps^\mp w\big|_{M_\lambda^\mu(k)}
\end{equation}
for all $\mu$ in a Zariski-dense subset of $P$ and $\lambda = \mu_+ \in P_+$.

\begin{lemma}\label{lem:proj_operator_replacement}
    For each $w \in W$, define the operator
    \[
        \calQ^{(\eps)}_\pm(w, k) = \sum_{j = 1}^d T_{u_j}(k \pm l)C_{\Delta_\eps^\mp w}(T_{q_j}(k))
    \]
    with the shorthand $C_{\Delta_\eps^\mp w}(T_{q_j}(k)) = (\Delta_\eps^\mp w) T_{q_j}(k) (\Delta_\eps^\mp w)^{-1}$. Let $\lambda \in P_+$ and $\mu \in W\lambda$. Then
    \[
        \calQ^{(\eps)}_\pm(w, k)\Delta_\eps^\mp w\big|_{M_\lambda^\mu(k)} = Q_\mu(k,k \pm l)\Delta_\eps^\mp w\big|_{M_\lambda^\mu(k)}.
    \]
    Here $Q_\mu(m,k) = \sum_{j = 1}^d q_j(r_m(\mu))T_{u_j}(k)$ extends the definition of $Q_\mu(k)$, i.e, $Q_\mu(k,k) = Q_\mu(k)$.
\end{lemma}

\begin{proof}
    The lemma follows from the observation
    \[
        C_{\Delta_\eps^\mp w}(T_{q_j}(k))(\Delta_\eps^\mp \phi^w) = q_j(r_\mu(k))\Delta_\eps^\mp \phi^w
    \]
    for all $\phi \in M_\lambda^\mu(k)$.
\end{proof}

Note that Lemma \ref{lem:proj_operator_replacement} almost establishes \eqref{eq:proj_operator_replacement} and that this is fulfilled when $Q_\mu(k,k \pm l) = Q_{\mu_{\eps,\pm}}(k \pm l)$. This happens when $r_k(\mu) = r_{k \pm l}(\mu_{\eps,\pm})$. Let us examine this condition more closely. Recall that $r_k(\mu) = \overline{v}(\mu)w_\lambda(\lambda + \rho_k)$ and note that $r_{k \pm l}(\mu_{\eps,\pm}) = \overline{v}(\mu_{\eps,\pm})w_{\lambda\mp\rho_l}(\lambda + \rho_k)$, whence
\begin{equation}\label{eq:eigenvalue_cond}
    r_k(\mu) = r_{k \pm l}(\mu_{\eps,\pm}) \iff \overline{v}(\mu)w_\lambda = \overline{v}(\mu_{\eps,\pm})w_{\lambda\mp\rho_l}
\end{equation}
because $\lambda + \rho_k$ is regular by our running assumption on $k$. We observe that \eqref{eq:eigenvalue_cond} defines a Zariski-dense set for $\mu$. This set includes all the cases when both $\mu$ and $\mu_{\eps,\pm}$ are regular since then $w_\lambda$ and $w_{\lambda\mp\rho_l}$ equal the identity and $\overline{v}(\mu) = \overline{v}(\mu_{\eps,\pm})$ as $\mu$ and $\mu_{\eps,\pm}$ lie inside the same Weyl chamber.

Rephrasing the condition using anti-dominant weights instead of dominant ones simplifies the condition slightly. Let $\mu_- = w_0 \lambda$ be the anti-dominant weight in the Weyl orbit of $\mu$ and denote by $v(\mu)$ the shortest element $w \in W$ such that $w\mu = \mu_-$. The relation to $\overline{v}(\mu)$ is given by \cite[(2.7.3)]{Mac03}, which states that $v(\mu)\overline{v}(\mu) = v(\lambda)$. Note that $v(\lambda) = w_0w_\lambda$, so that we arrive at $v(\mu) = w_0w_\lambda \overline{v}(\mu)^{-1}$. Thus, the condition is equivalent to $v(\mu) = v(\mu_{\eps,\pm})$. Geometrically, we can interpret the condition as $\mu_{\eps,\pm} = \mu \pm v(\mu)^{-1}\rho_l$ retaining the regularity of $\mu$.

Substituting the operators $\calQ^{(\eps)}_\pm(w, k)$ for $Q_{\mu_{\eps,\pm}}(k\pm l)$ into the expression \eqref{eq:fw_shift_operator_rewritten} yields an operator $\calG^{(\eps)}_\pm(k)$ that agrees with $\calG^{(\eps)}_\pm(\mu,k)$ on the Zariski-dense set given by the condition \eqref{eq:eigenvalue_cond}. In what follows, we compute the action of $\calG^{(\eps)}_\pm(k)$ outside this Zariski-dense set.

\begin{proposition}\label{prop:def_shift_operator}
    To each linear character $\eps$, we associate the $\eps$-forward/backward non-symmetric shift operator
    \[
        \calG^{(\eps)}_\pm(k) = \sum_{w\in W}\eps_\pm(w)\calQ^{(\eps)}_\pm(w,k) \Delta_\eps^\mp w.
    \]
    The operator $\calG^{(\eps)}_\pm(k)$ defines a map from $M(k)$ to $M(k \pm l)$. For $\lambda \in P_+$ and $\mu \in W\lambda$, the restriction of $\calG^{(\eps)}_\pm(k)$ to $M_\lambda^\mu(k)$ is given by
    \[
        \calG^{(\eps)}_\pm(k)\big|_{M_\lambda^\mu(k)} =
        \begin{cases}
            \calG^{(\eps)}_\pm(\mu,k)\big|_{M_\lambda^\mu(k)} & \text{if } v(\mu) = v(\mu_{\eps,\pm}); \\
            0 & \text{otherwise}.
        \end{cases}
    \]
\end{proposition}

\begin{proof}
    Using Lemma \ref{lem:proj_operator_replacement}, we find that
    \begin{equation}\label{eq:def_shift_operator_step}
        \calG^{(\eps)}_\pm(k)\phi = Q_\mu(k,k \pm l)(\Delta_\eps^\mp U_{\eps_\pm}\phi)
    \end{equation}
    for all $\phi \in M_\lambda^\mu(k)$. If $v(\mu) = v(\mu_{\eps,\pm})$, then $Q_\mu(k,k \pm l) = Q_{\mu_{\eps,\pm}}(k \pm l)$, as was argued before, so that $\calG^{(\eps)}_\pm(k)\phi = \calG^{(\eps)}_\pm(\mu,k)\phi$. Otherwise, suppose that $v(\mu) \neq v(\mu_{\eps,\pm})$ and consider the cases $\lambda \mp \rho_l \notin P_+$ and $\lambda \mp \rho_l \in P_+$ separately.
    
    In the first case, which only occurs when $\lambda - \rho_l \notin P_+$, we have that $U_\eps \phi = 0$. It follows that $\calG^{(\eps)}_\pm(k)\phi = 0 = \calG^{(\eps)}_\pm(\mu,k)\phi$, cf.~the proof of Proposition \ref{prop:fw_shift_factor}.
    
    Secondly, for the case $\lambda \mp \rho_l \in P_+$, we write $\Delta_\eps^\mp U_{\eps_\pm}\phi$ as a linear combination of the elements of the basis $(E_\nu(k \pm l))_{\nu\in W(\lambda \mp \rho_l)}$ of $M_{\lambda\mp\rho_l}(k)$, i.e., $\sum_{\nu \in W(\lambda \mp \rho_l)} c_\nu E_\nu(k \pm l)$. Using \eqref{eq:def_shift_operator_step}, we arrive at
    \begin{equation}\label{eq:def_shift_operator_step_2}
        \calG^{(\eps)}_\pm(k)\phi = \sum_{\nu \in W(\lambda \mp \rho_l)} c_\nu q(r_k(\lambda),r_{k \pm l}(\nu))E_\nu(k \pm l).
    \end{equation}
    We note that $r_{k \pm l}(\nu) = w_\nu r_k(\mu)$ with $w_\nu = \overline{v}(\nu)w_{\lambda\mp\rho_l}w_\lambda^{-1}\overline{v}(\mu)^{-1}$ and claim that $w_\nu \neq e$ for all $\nu \in W(\lambda \mp \rho_l)$. By the definition of the polynomial $q$, this implies that $q(r_k(\lambda),r_{k \pm l}(\nu)) = 0$, so that $\calG^{(\eps)}_\pm(k)\phi = 0$ as a result of \eqref{eq:def_shift_operator_step_2}.
    
    To prove our claim, we have to show that $\overline{v}(\mu)w_\lambda \notin W^{\lambda \mp \rho_l}w_{\lambda \mp \rho_l}$. If the opposite were true, then $\overline{v}(\mu)w_\lambda = w w_{\lambda \mp \rho_l}$ with $w \in W^{\lambda \mp \rho_l}$ and $\overline{v}(\mu)w_\lambda = \overline{v}(\mu_{\eps,\pm})w'$ with $w'\in W_{\lambda \mp \rho_l}$ by reason of the decomposition $W = W^{\lambda \mp \rho_l}W_{\lambda \mp \rho_l}$. However, this forces the contradiction $w = \overline{v}(\mu_{\eps,\pm})$ and $w' = w_{\lambda\mp\rho_l}$ since each element decomposes uniquely. This completes the proof of the proposition.
\end{proof}

We proceed to determine the shift factor of the operator $\calG^{(\eps)}_\pm(k)$. On the Zariski-dense set given by \eqref{eq:eigenvalue_cond}, it coincides with the shift factor $H^{(\eps)}_\pm(\mu, k)$ by Proposition \ref{prop:def_shift_operator}. Next, we seek to incorporate the vanishing of $\calG^{(\eps)}_\pm(k)$ outside the Zariski-dense set into the shift factor. In similar spirit, we substitute $\overline{v}(\mu)w_\lambda$ for $\overline{v}(\mu_{\eps,\pm})w_{\lambda\mp\rho_l}$ in $H^{(\eps)}_\pm(\mu, k)$, which turns out to capture the proper vanishing due to the second $\bfc$-function factor in $H^{(\eps)}_\pm(\mu, k)$.

\begin{lemma}[Vanishing Lemma 2]\label{lem:vanishing_2}
    Let $\lambda \in P_+$, $\mu \in W\lambda$, and suppose that $v(\mu) \neq v(\mu_{\eps,\pm})$. Then
    \[
        \bfc_{-\eps_\mp (k \pm l)}^+(\overline{v}(\mu)w_\lambda)(\lambda + \rho_k) = 0.
    \]
\end{lemma}

\begin{proof}
    By virtue of $\overline{v}(\mu)w_\lambda w_{\lambda \mp \rho_l} \notin W^{\lambda \mp \rho_l}$, see the proof of Proposition \ref{prop:def_shift_operator}, and \eqref{eq:shortest_coset_char}, there exists an $i \in I$ such that $\alpha_i \in R^0_{\lambda \mp \rho_l, +}$ and $(\overline{v}(\mu)w_\lambda)(\alpha_i) \in R^0_+$, so that $\alpha_i \in R^0_+ \cap (\overline{v}(\mu)w_\lambda)^{-1}R^0_+$. Note that $\eps_\mp(r_i) = 1$ in both cases since either $\eps_- = \triv$ or $\eps_+(W_{\lambda+\rho_l}) = \{1\}$. Therefore,
    \[
        (\lambda + \rho_k, \alpha_i^\vee) - \eps_\mp(r_i)(k^0(\alpha_i) \pm l(\alpha_i)) = (\lambda \mp \rho_l + \rho_{k \pm l}, \alpha_i^\vee) - (k^0(\alpha_i) \pm l(\alpha_i)) = 0,
    \]
    whence $\bfc_{-\eps_\mp (k \pm l)}^+(\overline{v}(\mu)w_\lambda)(\lambda + \rho_k) = 0$ because one of its factors vanishes.
\end{proof}

This establishes the following corollary.

\begin{corollary}\label{cor:adj_shift_factor}
    Let $\lambda \in P_+$ and $\mu \in W\lambda$. Writing
    \[
        \begin{aligned}
            \calH^{(\eps)}_\pm(\mu, k) &= \eps(\overline{v}(\mu))\pi(r_k(\mu))\bfc_{\eps_\pm k}^-(\overline{v}(\mu)w_\lambda)(\lambda + \rho_k) \\
            &\qquad\times\bfc_{-\eps_\mp (k \pm l)}^+(\overline{v}(\mu)w_\lambda)(\lambda + \rho_k),
        \end{aligned}
    \]
    we have
    \[
        \calH^{(\eps)}_\pm(\mu, k) =
        \begin{cases}
            H^{(\eps)}_\pm(\mu, k) & \text{if } v(\mu) = v(\mu_{\eps,\pm}); \\
            0 & \text{otherwise}.
        \end{cases}
    \]
\end{corollary}

We conclude this section with our main result by showing that the operator $\calG^{(\eps)}_\pm(k)$ is in fact a non-symmetric shift operator in the sense that it satisfies the transmutation property with the Dunkl--Cherednik operators.

\begin{theorem}\label{thm:transmutation_prop}
    The operator $\calG^{(\eps)}_\pm(k) : M(k) \to M(k \pm l)$ is a non-symmetric shift operator with shift $\pm l$, i.e., it satisfies the transmutation property
    \[
        \calG^{(\eps)}_\pm(k) T_p(k) = T_p(k \pm l) \calG^{(\eps)}_\pm(k), \qquad p \in S(\frakh).
    \]
    Furthermore, $\calG^{(\eps)}_\pm(k)$ shifts the multiplicity of the non-symmetric Heckman--Opdam polynomials by $\pm l$:
    \[
         \calG^{(\eps)}_\pm(k)E_\mu(k) = \calH^{(\eps)}_\pm(\mu, k)E_{\mu_{\eps,\pm}}(k \pm l), \qquad \mu \in P.
    \]
\end{theorem}

\begin{proof}
    The second property, stating that $\calH^{(\eps)}_\pm(\mu, k)$ is the shift factor of $\calG^{(\eps)}_\pm(k)$, is a direct consequence of Proposition \ref{prop:def_shift_operator} and Corollary \ref{cor:adj_shift_factor}. To prove the transmutation property, it suffices to show that $\calG^{(\eps)}_\pm(k) T_p(k)$ and $T_p(k \pm l) \calG^{(\eps)}_\pm(k)$ agree when applied to the basis $(E_\mu(k))_{\mu \in P}$ of $M(k)$. Applying the operators to $E_\mu(k)$, we recover multiples of $E_{\mu_{\eps,\pm}}(k \pm l)$ with coefficients
    \[
        p(r_k(\mu))\calH^{(\eps)}_\pm(\mu, k), \quad p(r_{k \pm l}(\mu_{\eps,\pm}))\calH^{(\eps)}_\pm(\mu, k).
    \]
    If $v(\mu) = v(\mu_{\eps,\pm})$, then the equality of these coefficients follows from the identity $r_k(\mu) = r_{k \pm l}(\mu_{\eps,\pm})$, and they remain equal when $v(\mu) \neq v(\mu_{\eps,\pm})$ since then both vanish by Corollary~\ref{cor:adj_shift_factor}.
\end{proof}

\begin{remark}
    The proof of Theorem \ref{thm:transmutation_prop} forces any non-symmetric shift operator to have the subspace spanned by the polynomials $E_\mu(k)$ with $r_k(\mu) \neq r_{k \pm l}(\mu_{\eps,\pm})$ in its kernel. This was observed for the operators $\calG^{(\eps)}_\pm(k)$ in Proposition \ref{prop:def_shift_operator}.
\end{remark}

\begin{example}\label{exa:rank_one}
    For the root system of type $BC_1$, the sign character $\eps$ is the only non-trivial linear character. In this case, the (non-)symmetric Heckman--Opdam polynomials can be identified with the (non-)symmetric Jacobi polynomials, up to normalization. Using the explicit form of the polynomial $q$ from Example \ref{exa:minimal_poly}, the non-symmetric shift operators become fully explicit after some elementary calculations:
    \[
        \begin{aligned}
            \calG^{(\eps)}_+(k) &= \frac{1}{e^{\epsilon_1} - e^{-\epsilon_1}} \partial_{\epsilon_1} - \frac{e^{\epsilon_1}(1 - r_1)}{(e^{\epsilon_1} - e^{-\epsilon_1})^2}; \\
            \calG^{(\eps)}_-(k) &= \partial_{\epsilon_1} + (k(\epsilon_1) + 2k(2\epsilon_1) - 1)(e^{\epsilon_1} + e^{-\epsilon_1}) + 2k(\epsilon_1) + e^{-\epsilon_1} - e^{\epsilon_1} r_1.
        \end{aligned}
    \]
    These operators agree with the non-symmetric shift operators from \cite[Prop.4.3]{vHvP24} and \cite[\S8.7]{OTL24}.
    
    The shift factors from Corollary \ref{cor:adj_shift_factor} simplify to
    \[
        \begin{aligned}
            \calG^{(\eps)}_\pm(k) E_{(n + 1)\epsilon_1}(k) &= (n + k^0(2\epsilon_1) \mp k^0(2\epsilon_1)) E_{(n + 1 \mp 1)\epsilon_1}(k \pm l); \\
            \calG^{(\eps)}_\pm(k) E_{-n\epsilon_1}(k) &= (n + k^0(2\epsilon_1) \mp k^0(2\epsilon_1)) E_{-(n \mp 1)\epsilon_1}(k \pm l)
        \end{aligned}
    \]
    for all $n \geq 0$. These relations agree with \cite[Lem.4.5]{vHvP24}. The multiplicity $k$ may be identified with the tuple $(k(\epsilon_1),k(2\epsilon_1))$, and then $\pm l$ corresponds to $\pm(0,1)$. 
    
    Finally, we note that non-symmetric shift operators with shifts $\pm(2,-1)$ were also found in \cite{vHvP24}, providing the analog of the symmetric shift operators with these shifts, see e.g.~\cite[Prop.3.3.1]{HeSc94}.
\end{example}


\subsection{Properties of non-symmetric shift operators}\label{sec:non_sym_shift_operators_prop}
This section addresses three properties of the non-symmetric shift operators. First, we observe that the order of the symmetric shift operator $G^{(\eps)}_+(k)$ is equal to $d_\eps = |\{\alpha \in R^0_+ \mid l(\alpha) = 1\}|$, whereas the `order' of the non-symmetric shift operator $\calG^{(\eps)}_\pm(k)$ is always given by $|R^0_+|$, regardless of the character $\eps$. This leads us to conjecture the existence of a fundamental non-symmetric shift operator $\tilde{\calG}^{(\eps)}_\pm(k)$ for $\eps$ also of `order' $d_\eps$. Second, we show that $\tilde{\calG}^{(\eps)}_+(k)$, provided it exists, restricts to a differential operator on the invariant subspace $M^W(k)$, precisely given by $G^{(\eps)}_+(k)$. Finally, the third property asserts that $\calG^{(\eps)}_-(k + l)$ is the adjoint operator of $\calG^{(\eps)}_+(k)$, reflecting the duality from their construction.

Rewriting the shift factor from Corollary \ref{cor:adj_shift_factor} provides us with some insights into the structure of the fundamental non-symmetric shift operator for $\eps$, which is used to formulate Conjecture~\ref{conj:fundamental_shift_operator}.

\begin{lemma}\label{lem:shift_factor_v2}
    Let $\lambda \in P_+$ and $\mu \in W\lambda$. Then
    \[
        \begin{aligned}
            \calH^{(\eps)}_\pm(\mu, k) &= \prod_{\alpha \in R^0_+, \thinspace l(\alpha) = 1}\left((\lambda + \rho_k, \alpha^\vee) \mp k^0(\alpha) - 1 + \delta_{\overline{v}(\mu) w_\lambda}(\alpha)\right) \\
            &\qquad \times \prod_{\alpha \in R^0_+, \thinspace l(\alpha) = 0}\left((r_k(\mu), \alpha^\vee) - k^0(\alpha)\right),
        \end{aligned}
    \]
    where $\delta_w(\alpha) = 1$ if $\alpha \in R^0_+ \cap w^{-1}R^0_-$ and $\delta_w(\alpha) = 0$ otherwise.
\end{lemma}

\begin{proof}
    The lemma follows from Corollary \ref{cor:adj_shift_factor} after expanding the definitions of the $\bfc$-functions, which makes it clear that $\pi(r_k(\mu))$ cancels their denominators and that the remaining quantity is precisely the desired expression.
\end{proof}

\begin{conjecture}\label{conj:fundamental_shift_operator}
    For any linear character $\eps$ of $W$, there exists a non-symmetric shift operator $\tilde{\calG}^{(\eps)}_\pm(k)$ with shift $\pm l$ that is fundamental in the following sense. All non-symmetric shift operators with shift $\pm l$ are of the form
    \[
        \tilde{\calG}^{(\eps)}_\pm(k)T_p(k) \qquad \text{for some } p \in S(\frakh).
    \]
    In particular, the non-symmetric shift operator $\calG^{(\eps)}_\pm(k)$ decomposes as follows:
    \[
        \calG^{(\eps)}_\pm(k) = \tilde{\calG}^{(\eps)}_\pm(k) \prod_{\alpha \in R^0_+, \thinspace l(\alpha) = 0}(T_{\alpha^\vee}(k) - k^0(\alpha)).
    \]
    Observe that $\calG^{(\eps)}_\pm(k) = \tilde{\calG}^{(\eps)}_\pm(k)$ for the sign character, so that the conjecture states that $\calG^{(\eps)}_\pm(k)$ is fundamental in this case.
\end{conjecture}

Conjecture \ref{conj:fundamental_shift_operator} is valid for the root system of type $BC_1$, see \cite[Thm.4.8]{vHvP24}. Additional evidence supporting this claim is the structure theorem for symmetric shift operators, see \cite[Thm.3.3.7]{HeSc94}, which states that the analog of Conjecture \ref{conj:fundamental_shift_operator} holds for all integral multiplicities $l$. The decomposition of $\calG^{(\eps)}_\pm(k)$ is consistent with the shift factor from Lemma \ref{lem:shift_factor_v2}.

In \cite[Thm.8.2(iii)]{OTL24}, it was observed that, for the sign character $\eps$, the (fundamental) non-symmetric shift operator $\calG^{(\eps)}_+(k)$ restricts to $G^{(\eps)}_+(k)$ on the invariant subspace $M^W(k)$. The following proposition reproves this fact and extends it to any $\eps$, assuming the validity of Conjecture~\ref{conj:fundamental_shift_operator}.

\begin{proposition}\label{prop:restriction}
    The operator $\tilde{\calG}^{(\eps)}_+(k)$ from Conjecture \ref{conj:fundamental_shift_operator}, i.e., $\calG^{(\eps)}_+(k)$ in case of the sign character, restricts to a differential operator on $M^W(k)$. This differential operator coincides with $G^{(\eps)}_+(k)$ from Theorem \ref{thm:sym_shift_factor}, so that
    \[
        \tilde{\calG}^{(\eps)}_+(k)P_\lambda(k) = h^{(\eps)}_+(\lambda, k) P_{\lambda - \rho_l}(k + l), \qquad \lambda \in P_+.
    \]
\end{proposition}

\begin{proof}
    To prove the proposition, it suffices to verify that $\tilde{\calG}^{(\eps)}_+(k)$ and $G^{(\eps)}_+(k)$  agree when applied to the basis $(P_\lambda(k))_{\lambda \in P_+}$ of $M^W(k)$. Theorem \ref{thm:transmutation_prop} and Lemma \ref{lem:shift_factor_v2} imply that
    \[
         \tilde{\calG}^{(\eps)}_+(k)E_\mu(k) = \tilde{\calH}^{(\eps)}_+(\mu, k)E_{\mu_{\eps,+}}(k + l)
    \]
    for all $\mu \in P$, where
    \[
        \tilde{\calH}^{(\eps)}_+(\mu, k) = \prod_{\alpha \in R^0_+, \thinspace l(\alpha) = 1}\left((\lambda + \rho_k, \alpha^\vee) - k^0(\alpha) - 1 + \delta_{\overline{v}(\mu) w_\lambda}(\alpha)\right).
    \]
    For $\lambda \in P_+$, we express $P_\lambda(k)$ as a linear combination of the polynomials $E_\mu(k)$ with $\mu \in W\lambda$ according to Lemma \ref{lem:eHO_in_nHO_basis_v2}, from which we obtain that the desired equation
    \[
        \tilde{\calG}^{(\eps)}_+(k)P_\lambda(k) = h^{(\eps)}_+(\lambda, k) P_{\lambda - \rho_l}(k + l) = G^{(\eps)}_+(k)P_\lambda(k)
    \]
    is equivalent to the following relation between the coefficients and the shift factors
    \begin{equation}\label{eq:restriction_step}
        \tilde{\calH}^{(\eps)}_+(\mu, k) \bfc_{-k}^+(\overline{v}(\mu)w_\lambda)(\lambda + \rho_k) = h^{(\eps)}_+(\lambda, k) \bfc_{-(k + l)}^+(\overline{v}(\mu_{\eps,+})w_{\lambda - \rho_l})(\lambda + \rho_k)
    \end{equation}
    for all $\mu \in W\lambda$.
    
    It remains to show the relation \eqref{eq:restriction_step}. It holds immediately in case that $\lambda - \rho_l \notin P_+$ because then both $\tilde{\calH}^{(\eps)}_+(\mu, k)$ and $h^{(\eps)}_+(\lambda, k)$ vanish by (the proof of) Lemma \ref{lem:vanishing_1}. Suppose that $\lambda \in \rho_l + P_+$. If $v(\mu) = v(\mu_{\eps,\pm})$, then \eqref{eq:restriction_step} can simply be verified by expanding the definitions and substituting\vspace{-0.1cm} $\overline{v}(\mu)w_\lambda = \overline{v}(\mu_{\eps,+})w_{\lambda - \rho_l}$. Finally, in case that $v(\mu) \neq v(\mu_{\eps,\pm})$, the factor $\tilde{\calH}^{(\eps)}_+(\mu, k)$ vanishes by Lemma \ref{lem:vanishing_2}, so that we have to show that the right-hand side of \eqref{eq:restriction_step} vanishes as well. Similar to the proof of Lemma \ref{lem:vanishing_2}, there exists an $i \in I$ such that the root $\alpha_i \in R^0_+ \cap (\overline{v}(\mu_{\eps,+})w_{\lambda - \rho_l})^{-1}R^0_+$ satisfies $\alpha_i \in R^0_\lambda$. Since $\eps(W_{0\lambda}) = \{1\}$, we have $\eps(r_i) = 1$, so that $l(\alpha_i) = 0$. Therefore,
    \[
        (\lambda + \rho_k, \alpha_i^\vee) - (k^0(\alpha_i) + l(\alpha_i)) = k^0(\alpha_i) - k^0(\alpha_i) = 0,
    \]
    whence $\bfc_{-(k + l)}^+(\overline{v}(\mu_{\eps,+})w_{\lambda - \rho_l})(\lambda + \rho_k)$ because one of its factors vanishes. This completes the proof of the proposition.
\end{proof}

Finally, in order to prove that $\calG^{(\eps)}_-(k + l)$ is the adjoint operator of $\calG^{(\eps)}_+(k)$, we initially show that the result holds for the fixed weights.

\begin{lemma}\label{lem:fw_adjoint}
    Let $\lambda \in P_+$ and $\mu \in W\mu$. Then
    \[
        \big(\calG^{(\eps)}_+(\mu, k) E_\mu(k), E_{\mu_{\eps,+}}(k + l)\big)_{k + l} = \big(E_\mu(k), \calG^{(\eps)}_-(\mu_{\eps,+}, k + l) E_{\mu_{\eps,+}}(k + l)\big)_k.
    \]
\end{lemma}

\begin{proof}
    The lemma holds trivially when $\lambda - \rho_l \notin P_+$ or $v(\mu) \neq v(\mu_{\eps,+})$ since then both sides of the identity vanish by Corollary \ref{cor:adj_shift_factor}. Suppose that $\lambda \in \rho_l + P_+$ and $v(\mu) = v(\mu_{\eps,+})$. Using that the Dunkl--Cherednik operators and $U_\eps$ are formally self-adjoint and that $\Delta_\eps^* = \Delta_\eps^{-1}$, we find that
    \[
        \big(\calG^{(\eps)}_+(\mu, k) E_\mu(k), E_{\mu_{\eps,+}}(k + l)\big)_{k + l} = \big(E_\mu(k), U_\eps\Delta_\eps Q_{\mu_{\eps,+}}(k + l) E_{\mu_{\eps,+}}(k + l)\big)_k.
    \]
    Note that $U_\eps\Delta_\eps = \Delta_\eps U_\triv$ and $r_{k + l}(\mu_{\eps,+}) = r_k(\mu)$, whence
    \[
        \begin{aligned}
            \big(\calG^{(\eps)}_+(\mu, k) E_\mu(k), E_{\mu_{\eps,+}}(k + l)\big)_{k + l} &= \big(E_\mu(k), \pi(r_k(\mu)) \Delta_\eps U_\triv E_{\mu_{\eps,+}}(k + l)\big)_k \\
            &= \big(E_\mu(k), Q_\mu(k) \Delta_\eps U_\triv E_{\mu_{\eps,+}}(k + l)\big)_k.
        \end{aligned}
    \]
    This proves the lemma since $\calG^{(\eps)}_-(\mu_{\eps,+}, k + l) = Q_\mu(k) \Delta_\eps U_\triv$. For the second equality, to introduce the $Q_\mu(k)$ we have used that
    \[
        \big(E_\mu(k), \pi(r_k(\mu))E_\nu(k)\big)_k = \delta_{\mu,\nu}\pi(r_k(\mu))\|E_\mu(k)\|_k^2 = \big(E_\mu(k), Q_\mu(k)E_\nu(k)\big)_k
    \]
    for all $\nu \in W\lambda$. This is a consequence of the orthogonality of the non-symmetric Heckman--Opdam polynomials and Lemma \ref{lem:proj_operator}.
\end{proof}

\begin{proposition}\label{prop:adjoint}
    The adjoint of the operator $\calG^{(\eps)}_+(k) : M(k) \to M(k + l)$ is given by $\calG^{(\eps)}_-(k + l) : M(k + l) \to M(k)$, i.e.,
    \[
        (\calG^{(\eps)}_+(k)\phi, \psi)_{k + l} = (\phi, \calG^{(\eps)}_-(k + l)\psi)_k, \qquad \phi \in M(k), \psi \in M(k + l).
    \]
\end{proposition}

\begin{proof}
    The proposition follows from Lemma \ref{lem:fw_adjoint} by expanding $\phi$ and $\psi$ in terms of $(E_\mu(k))_{\mu \in P}$ and $(E_\nu(k + l))_{\nu \in P}$, respectively, and using Proposition \ref{prop:def_shift_operator}.
\end{proof}


\subsection{$L^2$-norms of non-symmetric Heckman--Opdam polynomials}\label{sec:l2_norms}
In what is to come, we present an application of the non-symmetric shift operators to the $L^2$-norms of the non-symmetric Heckman--Opdam polynomials. In \cite{Opd89}, symmetric shift operators turned out to be instrumental in Opdam's resolution of the Macdonald constant term conjecture and the computation of the $L^2$-norms of the symmetric Heckman--Opdam polynomials.

From the adjoint relation from Proposition \ref{prop:adjoint}, we derive a recurrence relation between the $L^2$-norms with shifted multiplicities. The $L^2$-norm formulae are elegantly expressed in terms of the (generalized) Harish-Chandra $c$-functions:
\[
    \begin{aligned}
        \tilde{c}_w(\lambda, k) &= \prod_{\alpha \in R^0_+} \left(\frac{\Gamma\big((\lambda, \alpha^\vee) + \delta_w(\alpha)\big)}{\Gamma\big((\lambda, \alpha^\vee) + k^0(\alpha) + \delta_w(\alpha)\big)}\right); \\
        c^*_w(\lambda, k) &= \prod_{\alpha \in R^0_+} \left(\frac{\Gamma\big(-(\lambda, \alpha^\vee) - k^0(\alpha) + \delta_w(\alpha)\big)}{\Gamma\big(-(\lambda, \alpha^\vee) + \delta_w(\alpha)\big)}\right).
    \end{aligned}
\]
We recall the definition of $\delta_w(\alpha)$ from Lemma \ref{lem:shift_factor_v2}: $\delta_w(\alpha) = 1$ if $\alpha \in R^0_+ \cap w^{-1}R^0_-$ and $\delta_w(\alpha) = 0$ otherwise.

\begin{proposition}\label{prop:norms_nHO}
    Let $\lambda \in P_+$ be regular and $\mu \in W\lambda$. Suppose that $k \geq 0$ and $k - l \geq 0$. Then
    \[
        \frac{\|E_\mu(k)\|_k^2}{\|E_{\mu_{\eps,-}}(k - l)\|_{k - l}^2} = \frac{\tilde{c}_{\overline{v}(\mu) w_\lambda}(\lambda + \rho_k, k - l) c^*_{\overline{v}(\mu) w_\lambda}(-(\lambda + \rho_k), k)}{\tilde{c}_{\overline{v}(\mu) w_\lambda}(\lambda + \rho_k, k) c^*_{\overline{v}(\mu) w_\lambda}(-(\lambda + \rho_k), k - l)},
    \]
    and the right-hand side has neither zeros nor poles.
\end{proposition}

\begin{proof}
    A standard result for shift operators, see e.g.~\cite[Prop.3.5.1]{HeSc94}, that relies on Theorem \ref{thm:transmutation_prop} and Proposition~\ref{prop:adjoint} in our case yields the relation
    \begin{equation}\label{prop:norms_nHO_step}
        \frac{\|E_\mu(k)\|_k^2}{\|E_{\mu_{\eps,-}}(k - l)\|_{k - l}^2} = \frac{\calH^{(\eps)}_-(\mu, k)}{\calH^{(\eps)}_+(\mu_{\eps,-}, k - l)}.
    \end{equation}
    Using the functional equation $\Gamma(z + 1) = z\Gamma(z)$, we see that
    \[
        \begin{aligned}
            \frac{c^*_{\overline{v}(\mu) w_\lambda}(-(\lambda + \rho_k), k)}{c^*_{\overline{v}(\mu) w_\lambda}(-(\lambda + \rho_k), k + l)} &= \prod_{\alpha \in R^0_+, \thinspace l(\alpha) = 1}\big((\lambda + \rho_k, \alpha^\vee) - k^0(\alpha) - 1 + \delta_{\overline{v}(\mu) w_\lambda}(\alpha)\big); \\
            \frac{\tilde{c}_{\overline{v}(\mu) w_\lambda}(\lambda + \rho_k, k - l)}{\tilde{c}_{\overline{v}(\mu) w_\lambda}(\lambda + \rho_k, k)} &= \prod_{\alpha \in R^0_+, \thinspace l(\alpha) = 1}\big((\lambda + \rho_k, \alpha^\vee) + k^0(\alpha) - 1 + \delta_{\overline{v}(\mu) w_\lambda}(\alpha)\big),
        \end{aligned}
    \]
    which equal $\calH^{(\eps)}_+(\mu_{\eps,-}, k - l)\varpi(\mu,k)^{-1}$ and $\calH^{(\eps)}_-(\mu, k)\varpi(\mu,k)^{-1}$, respectively, with the common factor $\varpi(\mu,k)^{-1} = \prod_{\alpha \in R^0_+, \thinspace l(\alpha) = 0}\left((r_k(\mu), \alpha^\vee) - k^0(\alpha)\right)^{-1}$. Here we have used that $\overline{v}(\mu_{\eps,-})w_{\lambda + \rho_l} = \overline{v}(\mu)w_\lambda$ and $r_{k - l}(\mu_{\eps,-}) = r_k(\mu)$, which are valid due to the regularity of $\lambda$. Substituting these expressions into \eqref{prop:norms_nHO_step} proves the proposition as the common factor cancels. 
    
    The comment concerning the absence of zeros and poles of the right-hand side follows directly from the definition of the Harish--Chandra $c$-functions and the assumptions of the proposition.
\end{proof}

\begin{remark}
    The formula of Proposition \ref{prop:norms_nHO} remains valid for multiplicities $l$ obtained from positive integer combinations of the multiplicities coming from the characters $\eps$ by iterating the result. If the multiplicity $k$ is also of this form, then we obtain
    \[
        \|E_\mu(k)\|_k^2 = \frac{c^*_{\overline{v}(\mu) w_\lambda}(-(\lambda + \rho_k), k)}{\tilde{c}_{\overline{v}(\mu) w_\lambda}(\lambda + \rho_k, k)}
    \]
    by setting $l$ equal to $k$. This $L^2$-norm formula for the non-symmetric Heckman--Opdam polynomials extends to general positive multiplicities $k$ and is due to Opdam, see \cite[Thm.5.3]{Opd95}. The proof relied on the structure of certain unitary $\calH(k^0)$-modules, see \cite[Thm.4.1]{Opd95}, and a lemma \cite[Lem.5.2]{Opd95} reminiscent of the shift principle. Our approach with non-symmetric shift operators is more direct and closely resembles the original approach taken in \cite{Opd89} for the symmetric case.
\end{remark}


\section{Shift operators for non-symmetric Macdonald--Koornwinder polynomials}\label{sec:non_sym_shift_operators_q}\vspace{-0.1cm}
We commence by recalling some background on Macdonald--Koornwinder theory, based on \cite{Mac03}.\vspace{-0.05cm}


\subsection{Preliminaries on affine root systems}\label{sec:affine_root_systems}
Let $R$ be an irreducible root system on a real affine space $E$. Recall that $R$ consists of linear functions on $E$. There is a canonical construction that associates to $R$ an irreducible affine root system $S(R)$ that consists of affine-linear functions on $E$, see \cite[(1.2.1)]{Mac03}. Denote by $V$ the real $n$-dimensional vector space of translations of $E$. The vector space $V$ acts faithfully and transitively on $E$, so that each element of $E$ is of the form $x + v$ with $v \in V$ for some fixed $x \in E$. We equip $V$ with an inner product $<\cdot,\cdot>$. Each affine-linear function $f$ on $E$ is of the form
\[
    f(x + v) = f(x) \thinspace + <v, Df>,
\]
where $v \in V$ and $Df \in V$ is the gradient of $f$. Furthermore, the inner product of $V$ induces a positive semi-definite paring on affine-linear functions on $E$, also denoted $<\cdot,\cdot>$, by defining $<f,g> \thinspace = \thinspace <Df,Dg>$. The dual of $S(R)$ is given by $S(R)^\vee = \{a^\vee := 2a \thinspace / <a,a> \mid a \in S(R)\}$. 

The affine root system $S(R)$ is always reduced. All reduced irreducible affine root systems are similar to either $S(R)$ or $S(R)^\vee$. Thus, we can associate a type to each reduced irreducible affine root system $S$, given by the type of $R$. We decorate the type with $\vee$ in the case that $S \cong S(R)^\vee$. For example, if $R$ is the root system of type $A_n$, then $S(R)$ is of type $A_n$ and $S(R)^\vee$ is of type $A_n^\vee$. To a non-reduced irreducible affine root system $S$, we associate two reduced irreducible affine root systems
\[
    S_1 = \{a \in S \mid \tfrac{1}{2}a \notin S\}, \quad S_2 = \{a \in S \mid 2a \notin S\}.
\]
The type of $S$ is denoted by $(\text{type of } S_1, \text{type of } S_2)$. All non-reduced irreducible affine root systems are similar to subsystems of the affine root system of type $(C_n^\vee,C_n)$. Furthermore, the affine root system of type $(C_n^\vee,C_n)$ contains all affine root systems $S(R)$ and $S(R)^\vee$, where $R$ is one of the types $B_n$, $C_n$, $BC_n$, or $D_n$.

The Macdonald--Koornwinder polynomials are associated with a pair of irreducible affine root systems $(S,S')$ with root multiplicities $k : S \to \bbR$ and $k' : S' \to \bbR$ that are in duality and a deformation parameter $q \in (0,1)$. The pairs $(S,S')$ under consideration satisfy certain dualities which are described by two more pairs: a pair of irreducible root systems $(R,R')$ and a pair of lattices $(L,L')$. Writing $P$ and $P^\vee$ for the weight lattices and $Q$ and $Q^\vee$ for the root lattices of $R$ and $R^\vee$, respectively, there are the following three cases:
\begin{enumerate}[align=parleft]
    \item[\textbf{Case 1}:] $S = S(R)$, $S' = S(R^\vee)$, $R' = R^\vee$, $L = P$, and $L' = P^\vee$;
    \item[\textbf{Case 2}:] $S = S' = S(R)^\vee$, $R' = R$, and $L = L' = P^\vee$;
    \item[\textbf{Case 3}:] $S = S'$ is of type $(C_n^\vee,C_n)$, $R' = R$ is of type $C_n$, and $L = L' = Q^\vee$,
\end{enumerate}
where $R$ is any reduced irreducible root system in the first two cases. In each case, there is a bijection $a \mapsto a'$ between the roots of $S$ and $S'$. We fix positive systems $S^+ \subset S$ and $S'^+ \subset S'$ compatible with the bijection $a \mapsto a'$. This determines bases of simple roots $\{a_i\}_{i \in I}$ and $\{a_i'\}_{i \in I}$ for $S$ and $S'$ with $I = \{0,\dots,n\}$. This also induces positive systems $R^+ \subset R$ and $R'^+ \subset R'$, and bases of simple roots $\{\alpha_i\}_{i \in I_0}$ and $\{\alpha_i'\}_{i \in I_0}$ for $R$ and $R'$ with $I_0 = I \setminus \{0\}$. Furthermore, we denote by $W_S$ and $W_0$ the (affine) Weyl group of $S$ and $R$, respectively.

The dual multiplicity $k'$ of $k$ is the root multiplicity for $S'$ given by $k'(a') = k(a)$ in the first two cases. For the third case, we note that there are five (four if $n = 1$) $W_S$-orbits in $S$, so that $k$ is determined by five parameters $k_1,\dots,k_5$ since $k$ is constant on $W_S$-orbits. The dual multiplicity $k'$ is defined using parameters $k_1',\dots,k_5'$, which are certain linear combinations of $k_1,\dots,k_5$, see \cite[(1.5.1)]{Mac03}. Borrowing notation from \cite[Def.2.28]{Sch23}, we introduce the following collection of constants:
\[
    \tau_{a,k} = q^{\frac{1}{2}(k(a) + k(2a))}, \quad \tilde{\tau}_{a,k} = q^{\frac{1}{2}(k(a) - k(2a))}, \qquad a \in S,
\]
with the convention $k(2a) = 0$ when $2a \notin S$. Note that $\tau_{a,k} = \tilde{\tau}_{a,k}$ if $2a \notin S$. The constants $\tau_{a',k'}$ and $\tilde{\tau}_{a',k'}$ with $a' \in S'$ are defined similarly.

The Macdonald--Koornwinder polynomials are constructed as elements of the group algebra $A = K[L]$ of $L$ over a certain field $K$. The field $K$ is a subfield of $\bbR$ that contains the constants $\tau_{a,k}$, $\tilde{\tau}_{a,k}$, and $q_0 = q^{1/e}$, where $e$ is the natural number given by $<L,L'> = e^{-1}\bbZ$. The field $K$ also contains the constants $\tau_{a',k'}$ and $\tilde{\tau}_{a',k'}$. We equip $A$ with a non-degenerate pairing $(\cdot,\cdot)_k$ and denote by $A(k)$ the space $A$ together with this pairing. To define the pairing, we introduce the weight function
\[
    \Delta_{S,k} = \prod_{a \in S_1^+}(\tau_{a,k}\bfc_{a,k})^{-1},
\]
where
\[
    \bfc_{a,k} = \frac{(1 - \tau_{a,k}\tilde{\tau}_{a,k}e^a)(1 + \tau_{a,k}\tilde{\tau}_{a,k}^{-1}e^a)}{\tau_{a,k}(1 - e^{2a})}.
\]
The weight function $\Delta_{S,k}$ has a formal power series expansion $\Delta_{S,k} = \sum_{r \geq 0}\sum_{\lambda\in L}u_{\lambda,r}q^r e^\lambda$, which is used to define the constant term of $f\Delta_{S,k}$ with $f = \sum_{\lambda \in L}f_\lambda e^\lambda \in A$, as follows:
\[
    \ct(f\Delta_{S,k}) = \sum_{r \geq 0}\Big(\sum_{\lambda\in L}u_{\lambda,r}f_{-\lambda}\Big)q^r.
\]
The pairing is given in terms of the constant term by
\[
    (f,g)_k = \ct(fg^*\Delta_{S,k}),
\]
where $g^*$ is defined using the expansion $g = \sum_{\lambda \in L}g_\lambda e^\lambda$ by setting $g^* = \sum_{\lambda \in L}g_\lambda^*e^{-\lambda}$ with $g_\lambda^* \in K$ obtained from $g_\lambda$ by replacing $q_0$, $\tau_{a,k}$, and $\tilde{\tau}_{a,k}$ by their inverses.

Finally, we note that $W_0$ acts on $A(k)$ by transposition. For a linear character $\eps : W_0 \to K^\times$, we denote the $\eps$-isotypical component of $A(k)$ by $A^{(\eps)}(k)$ and we write $A_0(k)$ when $\eps$ is the trivial character.


\subsection{Orthogonal polynomials and the Hecke algebra}\label{sec:orthogonal_polynomials_q}
Recall the notation and definitions from Section \ref{sec:orthogonal_polynomials} and view them relative to $W_0$ and $L$ instead of $W$ and $P$. For example, $L_+$ denotes the cone of dominant weights and the notation relative to $W_0$ is decorated with an additional $0$, e.g., $W_{0\lambda}$ is stabilizer of $\lambda$ in $W_0$, $w_{0\lambda}$ denotes its longest element, and the partial ordering on $L$ is denoted by $<_0$.

For $\lambda \in L$, the non-symmetric Macdonald--Koornwinder polynomial $E_\lambda(k)$ is defined to be the unique element of $A(k)$ satisfying
\begin{enumerate}[(i)]
    \item $E_\lambda(k) = e^\lambda + \text{lower-order terms}$;
    \item $(E_\lambda(k), e^\mu)_k = 0 \text{ for all } \mu <_0 \lambda$.
\end{enumerate}
Similarly, for $\lambda \in L_+$ such that $\eps(W_{0\lambda}) = \{1\}$, we define the $\eps$-symmetric Macdonald--Koornwinder polynomial $P^{(\eps)}_\lambda(k)$ as the unique element of $A^{(\eps)}(k)$ satisfying
\begin{enumerate}[(i)]
    \item $P^{(\eps)}_\lambda(k) = m^{(\eps)}_\lambda + \text{lower-order terms}$;
    \item $(P^{(\eps)}_\lambda(k), m^{(\eps)}_\mu)_k = 0 \text{ for all } \mu <_0 \lambda$.
\end{enumerate}
If $\eps$ is the trivial character, then the polynomials $P_\lambda(k) := P_\lambda^{(\eps)}(k)$ are called the (symmetric) Macdonald--Koornwinder polynomials.

The analog of Lemma \ref{lem:sym_prop}, expressing $P^{(\eps)}_\lambda(k)$ as an $\eps$-symmetrization of $E_\lambda(k)$, is more involved. To state this properly, we introduce a $q$-analog of the $\eps$-symmetrizer using the Hecke algebra of $W_0$.

\begin{definition}\label{def:Hecke_algebra}
    The Hecke algebra $\frakH_0(k)$ of $W_0$ is the $K$-algebra with generators $T(w)$ for $w \in W_0$ subject to the relations
    \begin{enumerate}[(i)]
        \item $T(w)T(w') = T(ww') \text{ if } \ell(w) + \ell(w') = \ell(ww')$;
        \item $(T(r_i) - \tau_{a_i,k})(T(r_i) + \tau_{a_i,k}^{-1}) = 0 \text{ for all } i \in I_0$,
    \end{enumerate}
    where $\ell$ denotes the length function of $W_0$.
\end{definition}

There is a faithful representation of $\frakH_0(k)$ on $A(k)$, called the basic representation, see \cite[(4.3.3,10)]{Mac03}. A linear character $\eps$ of $W_0$ induces a linear character $\eps_k : \frakH_0(k) \to K^\times$ of $\frakH_0(k)$ defined by $\eps(T(w)) = \tau^{(\eps)}_{w,k}$, where $\tau^{(\eps)}_{w,k} = \tau^{(\eps)}_{r_{i_1},k} \cdots \tau^{(\eps)}_{r_{i_\ell},k}$ for a given reduced expression $w = r_{i_1} \cdots r_{i_\ell}$ and
\[
    \tau^{(\eps)}_{r_i,k} = 
    \begin{cases}
        \tau_{a_i,k} & \text{if } \eps(r_i) = 1; \\
        -\tau_{a_i,k}^{-1} & \text{otherwise}.
    \end{cases}
\]
For the trivial character, we simply write $\tau_{w,k}$. The $\eps$-symmetrizer is defined as
\[
    U_\eps(k) = \eps(w_0)\big(\tau_{w_0,k}^{(\eps)}\big)^{-1}\sum_{w\in W_0}\tau^{(\eps)}_{w,k}T(w),
\]
and it maps $A(k)$ onto the $\eps_k$-isotypical component $A^{(\eps_k)}(k)$ of $A(k)$. The isotypical components with respect to $\eps_k$ and $\eps$ coincide, i.e., $A^{(\eps_k)}(k) = A^{(\eps)}(k)$, which can be inferred from the relations \cite[(4.3.12),(5.8.7)]{Mac03}.

This brings us to the appropriate analog of Lemma \ref{lem:sym_prop}.

\begin{lemma}\label{lem:sym_prop_q}
    Let $\lambda \in L_+$ such that $\eps(W_{0\lambda}) = \{1\}$. Then
    \[
        P^{(\eps)}_\lambda(k) = \tau_{w_0,k}W_{0\lambda}(\tau_k^2)^{-1}U_\eps(k) E_\lambda(k),
    \]
    where $W_{0\lambda}(\tau_k^2) = \sum_{w \in W_{0\lambda}}\tau_{w,k}^2$ is the Poincaré polynomial of $W_{0\lambda}$ evaluated at the function $\tau_k^2 : W_{0\lambda} \to K, \thinspace w \mapsto \tau_{w,k}^2$.
\end{lemma}

\begin{proof}
    This follows from \cite[(5.7.7)]{Mac03}, taking into account that our conventions for $P^{(\eps)}_\lambda(k)$ and $U_\eps(k)$ differ from \cite{Mac03}. First, $P^{(\eps)}_\lambda(k)$ is normalized such that the coefficient of $e^\lambda$ is $1$ instead of normalizing the coefficient of $e^{w_0\lambda}$, giving an additional factor of $\eps(w_0)$. Second, this additional factor is included in the definition of $U_\eps(k)$, cf.~\cite[5.5.6]{Mac03}.
\end{proof}


\subsection{Cherednik operators and the affine Hecke algebra}\label{sec:affine_hecke_algebra}
The group $W_S$ equals the semi-direct product $W_S = W_0 \ltimes t(Q^\vee)$ with $t(Q^\vee) = \{t(\lambda) : x \mapsto x + \lambda \mid \lambda \in Q^\vee\}$. The lattice $Q^\vee$ is a sublatice of $L'$, and in order to define the affine Hecke algebra, we introduce the extended affine Weyl group $W = W_0 \ltimes t(L')$. The group $W$ contains $W_S$ as a normal subgroup and $W = W_S \rtimes \Omega'$, where $\Omega' = L'/Q^\vee$ is a finite abelian group. The length function $\ell$ of $W_S$ can be extended to $W$, see \cite[\S2.2]{Mac03}, and $\Omega'$ is isomorphic to the subgroup $\{u \in W \mid \ell(u) = 0\}$.

The affine Hecke algebra $\frakH(k)$ is defined to be the Hecke algebra of $W$ from Definition \ref{def:Hecke_algebra}. The basic representation of $\frakH_0(k)$ extends to a faithful representation $\beta_k : \frakH(k) \to \End(A(k))$ of $\frakH(k)$ on $A(k)$, see \cite[(4.3.10)]{Mac03}. For $\lambda' \in L'$, we can write $\lambda' = \mu' - \nu'$ with $\mu',\nu' \in L'_+$ which we use to define the element $Y^{\lambda'} = T(t(\mu'))T(t(\nu'))^{-1}$. The polynomials $E_\lambda(k)$ are eigenfunctions of the Cherednik operators $Y^{\lambda'}(k) = \beta_k(Y^{\lambda'})$, viz.,
\[
    Y^{\lambda'}(k)E_\mu(k) = q^{-<\lambda',r_{k'}(\mu)>}E_\mu(k),
\]
where $r_{k'}(\mu) = \mu - v(\mu^{-1})\rho_{k'}$ and $\rho_{k'} = \frac{1}{2}\sum_{\alpha \in R^+}k'(\alpha^\vee)\alpha$. To each element $f' \in A' = K[L']$, we can associate an element $Y^{f'}$ and an operator $Y^{f'}(k)$, whose action on the polynomials $E_\mu(k)$ is given by
\[
    Y^{f'}(k)E_\mu(k) = f'(-r_{k'}(\mu))E_\mu(k),
\]
where the evaluation $f'(x)$ is obtained from $f' = \sum_{\lambda' \in L'} f_{\lambda'}e^{\lambda'}$ by replacing $e^{\lambda'}$ with $q^{<\lambda',x>}$. Furthermore, the operators $Y^{f'}(k)$ for $f' \in A'_0 := (A')^{W_0}$ act on the $\eps$-symmetric polynomials $P^{(\eps)}_\lambda(k)$ as follows:
\[
    Y^{f'}(k)P^{(\eps)}_\lambda(k) = f'(-\mu - \rho_{k'})P^{(\eps)}_\lambda(k).
\]
This is a direct consequence of the action of $Y^{f'}(k)$ on the polynomials $E_\mu(k)$ and the equality $f'(-\mu - \rho_{k'}) = f'(-r_{k'}(\mu))$.

For $\lambda \in L_+$, we define the subspace $A_\lambda(k) \subset A(k)$ as the $K$-linear span of the polynomials $E_\mu(k)$ with $\mu \in W_0\lambda$, which is isomorphic to $K[W_0/W_{0\lambda}]$ as a vector space. The space $A_\lambda(k)$ is an $\frakH(k)$-module with central character $-r_{k'}(\lambda)$, since the center of $\frakH(k)$ coincides with $A'_0(Y) = \{Y^{f'}(k) \mid f' \in A'_0\}$. For $\mu \in W_0\lambda$, the weight space
\[
    A_\lambda^\mu(k) = \big\{f \in A_\lambda(k) \mid Y^{\lambda'}(k)f = q^{-<\lambda',r_{k'}(\mu)>}f, \quad \lambda'\in L'\big\}
\]
is one-dimensional and spanned by $E_\mu(k)$.

For each character $\eps$, the module $A_\lambda(k)$ may contain an $\eps$-isotypical component $A_\lambda^{(\eps)}(k)$ which is at most one-dimensional. The following lemma states two equivalent conditions that capture when this occurs. For the lemma, we need the following notation. Let
\[
    S_0 = \{a \in S \mid a(0) = 0\}, \quad S_{01} = S_0 \cap S_1, \quad S_{01}^+ = S_{01} \cap S^+,
\]
define the root multiplicity
\[
    l(a) =
    \begin{cases}
        1 & \text{if } \eps(r_a) = -1 \text{ and } a \in S_0; \\
        0 & \text{otherwise},
    \end{cases}
\]
and, as in \cite[(5.8.9)]{Mac03}, put $\tilde{\rho}_l = \frac{1}{2}\sum_{a \in S_{01}^+}l(a)u_aa \in L$ with $u_a = 2$ if $2a \in S$ and $u_a = 1$ otherwise.

\begin{remark}\label{rmk:shift_mult}
    It will be convenient to rewrite $\tilde{\rho}_l$ in terms of $\rho_{k'}$ for a certain multiplicity $k'$ that depends on the different cases. Explicitly, we have
    \begin{enumerate}[align=parleft]
        \item[\textbf{Case 1}] $\tilde{\rho}_l = \rho_{l'}$, which is immediate since $S_{01}^+ = R$;
        \item[\textbf{Case 2}] $\tilde{\rho}_l = \rho_{{l^\wedge}'}$, expressed in terms of the root multiplicity $l^\wedge(a) = 2l(a)/|a|^2$ for $a \in S$, which stems from the fact that $S_{01}^+ = R^\vee$ instead of $S_{01}^+ = R$;
        \item[\textbf{Case 3}] $\tilde{\rho}_l = \rho_{l'}$, which follows from a short computation using the explicit expression for $\rho_{k'}$ in terms of $k_1'$ and $k_5'$ from \cite[\S1.5]{Mac03}.
    \end{enumerate}
    In order to have uniform notation, we write $l^\wedge$ for the root multiplicity satisfying $\tilde{\rho}_l = \rho_{{l^\wedge}'}$. 
\end{remark}

\begin{lemma}\label{lem:dim_equivalence_q}
    For $\lambda \in L_+$, the following are equivalent:
    \[
        \dim A_\lambda^{(\eps)}(k) = 1 \iff \eps(W_{0\lambda}) = \{1\} \iff \lambda \in \tilde{\rho}_l + L_+,
    \]
    and if these conditions are satisfied then $A_\lambda^{(\eps)}(k)$ is spanned by $P^{(\eps)}_\lambda(k)$.
\end{lemma}

\begin{proof}
    Note that $A_\lambda^{(\eps)}(k) = U_\eps(k)A_\lambda(k)$, so that $P^{(\eps)}_\lambda(k) \in A_\lambda^{(\eps)}(k)$ by Lemma \ref{lem:sym_prop_q}. Since $\dim A_\lambda^{(\eps)}(k) \leq 1$ by \cite[\S5.7]{Mac03}, it suffices to show that
    \[
        \big(U_\eps(k)E_\mu(k) \neq 0 \qquad \text{for some } \mu \in W_0\lambda\big) \iff \eps(W_{0\lambda}) = \{1\}.
    \]
    The right-to-left implication follows from \cite[(5.7.7)]{Mac03}, which states that $U_\eps(k)E_\lambda(k) \neq 0$. For the other direction, assuming that $\eps(W_{0\lambda}) \neq \{1\}$ provides the existence of some $i \in I_0$ such that $r_i \in W_{0\lambda}$ and $\eps(r_i) = -1$. This implies that $U_\eps(k)E_\mu(k) = 0$ for all $\mu \in W_0\lambda$ by \cite[(5.7.1,2)]{Mac03}.

    We prove the remaining equivalence by showing that $\lambda - \tilde{\rho}_l \notin L_+$ and $\eps(W_{0\lambda}) \neq \{1\}$ are equivalent using a chain of equivalences. Note that $\lambda - \tilde{\rho}_l \notin L_+$ is equivalent with
    \begin{equation}\label{eq:equiv_1}
        (\lambda,\alpha_i') - (\tilde{\rho}_l,\alpha_i') \in \bbZ_{\leq -1} \qquad \text{for some } i \in I_0.
    \end{equation}
    In each case, it can be verified that
    \begin{equation}\label{eq:equiv_step}
        (\tilde{\rho}_l,\alpha_i') = v_{\alpha_i}(\rho_{l'},\alpha_i^\vee) = v_{\alpha_i}l'(\alpha_i^\vee) \in v_{\alpha_i}\bbZ_{\geq 0},
    \end{equation}
    using \cite[(1.5.4)]{Mac03} for the second equality. Here $v_\alpha = 2$ if $\frac{1}{2}\alpha \in S$ and $v_a = 1$ otherwise, which becomes relevant in the third case when $\alpha_n = 2\epsilon_n$ is the long root of $R$. Upon observing that $(\lambda,\alpha_i') \in v_{\alpha_i}\bbZ_{\geq 0}$ and using \eqref{eq:equiv_step}, we find that \eqref{eq:equiv_1} is equivalent with
    \begin{equation}\label{eq:equiv_2}
        (\lambda,\alpha_i') = 0, \qquad l'(\alpha_i^\vee) = 1 \qquad \text{for some } i \in I_0.
    \end{equation}
    Finally, \eqref{eq:equiv_2} is equivalent with $r_i \in W_{0\lambda}$ and $\eps(r_i) = -1$ for some $i \in I_0$, i.e., $\eps(W_{0\lambda}) \neq \{1\}$.
\end{proof}

Since the polynomial $P^{(\eps)}_\lambda(k)$ is an element of $A_\lambda(k)$, it can be expressed as a linear combination of the elements of the basis $(E_\mu(k))_{\mu \in W_0\lambda}$ of $A_\lambda(k)$, cf.~Lemma \ref{lem:eHO_in_nHO_basis}. The coefficients are described in terms of the $\bfc$-functions
\[
    \bfc_{k'}^\pm(w)(\cdot) = \prod_{a \in S_1'^+ \cap w^{-1}S_1'^\pm}\bfc_{a',k'}(\cdot), \qquad w \in W_0,
\]
where the evaluation $\bfc_{a',k'}(x)$ is obtained from $\bfc_{a',k'}$ by replacing $e^{a'}$ with $q^{<a',x>}$.

\begin{assumption}
    To avoid division by zero after evaluating the $\bfc$-functions, we follow \cite[p.100:($\ast$)]{Mac03} and impose from this point on the following genericity condition on the multiplicity $k$: the vector $\rho_{k'}$ is not fixed by any element of $W' = W_0 \ltimes t(L)$, except for the identity, ensuring that $\lambda + \rho_{k'}$ and $r_{k'}(\lambda)$ are regular for all $\lambda \in L$.
\end{assumption}

\begin{lemma}\label{lem:eMK_in_nMK_basis}
    Let $\lambda \in \tilde{\rho}_l + L_+$. Then
    \[
        P_\lambda^{(\eps)}(k) = \sum_{\mu \in W_0\lambda}\eps(\overline{v}(\mu))\tau_{\overline{v}(\mu)w_{0\lambda},k}^{-1}\tau_{w_0,k}\bfc_{-\eps k'}^+(\overline{v}(\mu)w_{0\lambda})(\lambda + \rho_{k'})E_\mu(k).
    \]
\end{lemma}

\begin{proof}
    In \cite[(5.6.7),(5.7.8)]{Mac03}, it is shown that
    \[
        P_\lambda^{(\eps)}(k) = \eps(w_0)\sum_{\mu \in W_0\lambda}\eps(v(\mu))\tau_{v(\mu),k}\bfc_{-\eps k'}^-(v(\mu))(r_{k'}(\mu))E_\mu(k).
    \]
    We observe that $S_1'^+\cap v(\mu)^{-1}S_1'^- = \overline{v}(\mu)w_{0\lambda}(S_1'^+\cap(\overline{v}(\mu)w_{0\lambda})^{-1}S_1'^+)$, which is due to the equality $v(\mu) = w_0w_{0\lambda}\overline{v}(\mu)^{-1}$, and using this we find that
    \[
        \bfc_{-\eps k'}^-(v(\mu))(r_{k'}(\mu)) = \bfc_{-\eps k'}^+(\overline{v}(\mu)w_{0\lambda})(\lambda+\rho_{k'}).
    \]
    Here we have also used that $r_{k'}(\mu)=\overline{v}(\mu)w_{0\lambda}(\lambda + \rho_{k'})$. The desired expressions follows from this upon substituting $\eps(v(\mu))\eps(w_0) = \eps(\overline{v}(\mu))$ and $\tau_{v(\mu),k} = \tau_{\overline{v}(\mu)w_{0\lambda},k}^{-1}\tau_{w_0,k}$.
\end{proof}

The $\bfc$-function also serves to connect $U_\eps(k) E_\mu(k)$ and $P^{(\eps)}_\lambda(k)$, giving the analog of Lemma \ref{lem:relating_U_nHO_and_eHO}. For this purpose, we first consider the analog of Lemma \ref{lem:poincare_poly}.

\begin{lemma}\label{lem:poincare_poly_q}
    For $\lambda \in L_+$, we have $\bfc_{k'}^-(w_{0\lambda})(\lambda + \rho_{k'}) = \tau_{w_{0\lambda},k'}^{-1}W_{0\lambda}(\tau_{k'}^2)$.
\end{lemma}

\begin{proof}
    The identity \cite[(5.1.36)]{Mac03}, written in terms of $\bfc$-functions using \cite[(5.1.3)]{Mac03}, applied to the affine root system $S'_\lambda = \{a' \in S' \mid \thinspace <a',\lambda> \thinspace = 0\}$ reads
    \begin{equation}\label{eq:poincare_poly_q_step}
        \sum_{w\in W_{0\lambda}} \prod_{a' \in (S_\lambda')_{01}^+} \bfc_{w(a'),k'} = \tau_{w_{0\lambda},k'}^{-1}W_{0\lambda}(\tau_{k'}^2).
    \end{equation}
    Evaluating $\bfc_{w(a'),k'}$ at $\lambda + \rho_{k'}$ gives
    \[
        \bfc_{w(a'),k'}(\lambda + \rho_{k'}) = \frac{(1 - q^{k'(a') + <w(a'),\rho_{k'}>})(1 + q^{k'(2a') + <w(a'),\rho_{k'}>})}{\tau_{a',k'}(1 - q^{2<w(a'),\rho_{k'}>})}.
    \]
    For $w \in W_{0\lambda}\setminus\{e\}$, there exists a root $a' \in (S_\lambda')_{01}^+$ such that $w(a') = -a_i'$ for some $i \in I_0$. Therefore,
    \[
        <w(a'),\rho_{k'}> \thinspace = -<a_i',\rho_{k'}> \thinspace = -k'(a_i) = -k'(a'),
    \]
    which implies that $q^{k'(a') \thinspace + \thinspace <w(a'),\rho_{k'}>} = 1$, showing that $\bfc_{w(a'),k'}(\lambda + \rho_{k'}) = 0$. As a result of this vanishing, evaluating \eqref{eq:poincare_poly_q_step} at $\lambda + \rho_{k'}$ yields
    \[
        \prod_{a' \in (S_\lambda')_{01}^+} \bfc_{a',k'} = \tau_{w_{0\lambda},k'}^{-1}W_{0\lambda}(\tau_{k'}^2).
    \]
    The left-hand side of this equation is exactly $\bfc_{k'}^-(w_{0\lambda})(\lambda + \rho_{k'})$ because $S_1'^+ \cap w_{0\lambda}^{-1}S_1'^- = (S_\lambda')_{01}^+$, which concludes the proof of the lemma.
\end{proof}

\begin{lemma}\label{lem:relating_U_nMK_and_eMK}
    Let $\lambda \in \tilde{\rho}_l + L_+$ and $\mu \in W_0\lambda$. Then
    \[
        U_\eps(k) E_\mu(k) = \eps(\overline{v}(\mu))\tau_{\overline{v}(\mu)w_{0\lambda},k}\tau_{w_0,k}^{-1}\bfc_{\eps k'}^-(\overline{v}(\mu)w_{0\lambda})(\lambda + \rho_{k'})P^{(\eps)}_\lambda(k).
    \]
\end{lemma}

\begin{proof}
    Specializing \cite[Lem.4.5]{Sch23} to our situation gives us
    \[
        U_\eps(k)E_\mu(k) = \eps(\overline{v}(\mu))\tau_{\overline{v}(\mu),k}\bfc_{\eps k'}^-(\overline{v}(\mu))(r_{k'}(\lambda))U_\eps(k)E_\lambda(k).
    \]
    The equality $\ell(\overline{v}(\mu)w_{0\lambda}) = \ell(\overline{v}(\mu)) + \ell(w_{0\lambda})$ implies the decomposition
    \[
        S_1'^+\cap(\overline{v}(\mu)w_{0\lambda})^{-1}S_1'^- = w_{0\lambda}^{-1}(S_1'^+\cap\overline{v}(\mu)^{-1}S_1'^-)\sqcup(S_1'^+\cap w_{0\lambda}^{-1}S_1'^-)
    \]
    by \cite[(2.2.4)]{Mac03}, whence
    \begin{equation}\label{eq:relating_U_nMK_and_eMK_step_1}
        \bfc_{\eps k'}^-(\overline{v}(\mu)w_{0\lambda})(\lambda+\rho_{k'}) = \bfc_{k'}^-(\overline{v}(\mu))(r_{k'}(\lambda))\bfc_{\eps k'}^-(w_{0\lambda})(\lambda+\rho_{k'}).
    \end{equation}
    Note that we have replaced $\eps k'$ with $k'$, which is valid because $\eps k' \equiv k'$ on $S_1'^+ \cap w_{0\lambda}^{-1}S_1'^- = (S_\lambda')_{01}^+$. Applying \cite[(5.1.41)]{Mac03} to $W_{0\lambda}$ results in $W_{0\lambda}(\tau_{k'}^2) = W_{0\lambda}(\tau_k^2)$. Together with the equality $\tau_{w_{0\lambda},k'} = \tau_{w_{0\lambda},k}$, we arrive at
    \begin{equation}\label{eq:relating_U_nMK_and_eMK_step_2}
        \tau_{w_{0\lambda},k'}^{-1}W_{0\lambda}(\tau_{k'}^2) = \tau_{w_{0\lambda},k}^{-1}W_{0\lambda}(\tau_k^2).
    \end{equation}
    The relation $\tau_{w,k'} = \tau_{w,k}$ holds for all $w \in W_0$, so it applies to $w_{0\lambda}$ in particular. Note that this relation follows immediately from the equality $\tau_{a',k'} = \tau_{a,k}$ for all $a \in S_0$. Combining \eqref{eq:relating_U_nMK_and_eMK_step_1}, \eqref{eq:relating_U_nMK_and_eMK_step_2}, and Lemma \ref{lem:poincare_poly_q}, we find that
    \begin{equation}\label{eq:relating_U_nMK_and_eMK_step_3}
        \tau_{\overline{v}(\mu),k}\bfc_{\eps k'}^-(\overline{v}(\mu))(r_{k'}(\lambda)) = \tau_{\overline{v}(\mu)w_{0\lambda},k}W_{0\lambda}(\tau_k^2)^{-1}\bfc_{\eps k'}^-(\overline{v}(\mu)w_{0\lambda})(\lambda + \rho_{k'}).
    \end{equation}
    Recall from Lemma \ref{lem:sym_prop_q} that $P^{(\eps)}_\lambda(k) = \tau_{w_0,k}W_{0\lambda}(\tau_k^2)^{-1}U_\eps(k) E_\lambda(k)$, from which the lemma follows upon substituting \eqref{eq:relating_U_nMK_and_eMK_step_3} into the expression for $U_\eps(k)E_\mu(k)$.
\end{proof}


\subsection{Symmetric $q$-shift operators}\label{sec:sym_shift_operators_q}
The symmetric $q$-shift operators relate the symmetric Mac-donald--Koornwinder polynomials for different multiplicities. We restate the construction of the $\eps$-forward $q$-shift operators from \cite{Mac03} in such a manner as to emphasize the usage of the $q$-analog of the shift principle (Proposition \ref{prop:shift_prin}).

To this end, for $a \in S_1$ we put
\[
    \delta_{a,k} =
    \begin{cases}
        (e^{a/2} - e^{-a/2})\bfc_{a,k} & \text{if } 2a \notin S; \\
        (e^a - e^{-a})\bfc_{a,k} & \text{otherwise},
    \end{cases}
\]
and define
\[
    \delta_{\eps,k} = \prod_{a \in S_{01}^+, \thinspace l(a) = 1}\delta_{a,k}, \qquad \delta'_{\eps,k'} = \prod_{a' \in S_{01}'^+, \thinspace l'(a') = 1}\delta_{a',k'}.
\]

\begin{proposition}[$q$-Shift principle]\label{prop:shift_prin_q}
    Let $\lambda \in \tilde{\rho}_l + L_+$. Then
    \[
        P_{\lambda - \tilde{\rho}_l}(k + l^\wedge) = q^{-\frac{1}{2}k\cdot l}\delta_{\eps,k}^{-1}P^{(\eps)}_\lambda(k),
    \]
    with $k\cdot l = \sum_{a \in S_0^+}k(a)l(a)$.
\end{proposition}

\begin{proof}
    The lemma is essentially \cite[(5.8.9)]{Mac03}, correcting for the different normalization conventions for $P^{(\eps)}_\lambda(k)$, which removes the factor $\eps(w_0)$, and for an omission that is only present in the second case. This omission originates in \cite[(5.8.4)]{Mac03}, where the formula should read
    \begin{equation}\label{eq:5.8.4_corrected}
        \delta_{\eps,k} \delta_{\eps,k}^* \Delta_{S,k} = \nabla_{S, k + l^\wedge}/\Delta_{S,\eps k}^0,
    \end{equation}
    using the multiplicity $l^\wedge$ from Remark \ref{rmk:shift_mult}. In fact, the multiplicity $l^\wedge$ that was designed to resolve this issue in the second case.
\end{proof}

The $\eps$-forward symmetric $q$-shift operator $G^{(\eps)}_+(k)$ is defined as
\[
    G^{(\eps)}_+(k) = \delta_{\eps,k}^{-1}\delta'_{\eps,k'}(Y(k)^{-1}),
\]
and a backward variant is defined similarly, see~\cite[\S5.9]{Mac03}.

We revisit the proof of \cite[(5.9.7)]{Mac03}, making the role of the $q$-shift principle more explicit. The adjusted proof also makes clear how the omission in \cite[(5.8.4)]{Mac03}, see \eqref{eq:5.8.4_corrected}, affects the statement and its proof.

\begin{theorem}\label{thm:sym_shift_factor_q}
    The operator $G^{(\eps)}_+(k)$ shifts the multiplicity of the symmetric Macdonald--Koorn-winder polynomials by $l^\wedge$:
    \[
        G^{(\eps)}_+(k)P_\lambda(k) = h^{(\eps)}_+(\lambda,k) P_{\lambda - \tilde{\rho}_l}(k + l^\wedge), \qquad \lambda \in L_+,
    \]
    with $h^{(\eps)}_+(\lambda,k) = q^{\frac{1}{2}k\cdot l}\delta'_{\eps,-k'}(\lambda + \rho_{k'})$.
\end{theorem}

\begin{proof}
    According to \cite[(5.9.5)]{Mac03}, we have $\delta'_{\eps,k'}(Y(k)^{-1}) U_\triv(k) = \eps(w_0)U_\eps(k)\delta'_{\eps,k'}(Y(k))$, where the additional factor of $\eps(w_0)$ comes from the definition of $U_\eps(k)$. Using this relation and Lemma~\ref{lem:sym_prop_q}, we arrive at
    \[
        \delta'_{\eps,k'}(Y(k)^{-1})P_\lambda(k) = \eps(w_0)\tau_{w_0,k}W_{0\lambda}(\tau_k^2)^{-1}U_\eps(k) \delta'_{\eps,k'}(Y(k))E_\lambda(k).
    \]
    This shows that $G^{(\eps)}_+(k)P_\lambda(k) = 0$ when $\lambda - \tilde{\rho}_l \notin L_+$ since then $U_\eps(k)E_\lambda(k) = 0$, which is in agreement with the theorem because $P_{\lambda - \tilde{\rho}_l}(k + l^\wedge) = 0$ in this case. Otherwise, we suppose that $\lambda \in \tilde{\rho}_l + L_+$. We have that
    \[
         \delta'_{\eps,k'}(Y(k))E_\lambda(k) = \delta'_{\eps,k'}(-r_{k'}(\lambda))E_\lambda(k)
    \]
    and $\delta'_{\eps,k'}(-r_{k'}(\lambda)) = \eps(w_0)\delta'_{\eps,-k'}(\lambda + \rho_{k'})$, which follows from $r_{k'}(\lambda) = w_{0\lambda}(\lambda + \rho_{k'})$ and the observation that $w_{0\lambda}$ permutes all roots $a' \in S_{01}'^+$ satisfying $l'(a') = 1$ due to $\eps(W_{0\lambda}) = \{1\}$. Therefore,
    \[
        \begin{aligned}
            G^{(\eps)}_+(k)P_\lambda(k) &= \delta'_{\eps,-k'}(\lambda + \rho_{k'})\delta_{\eps,k}^{-1}P_\lambda^{(\eps)}(k) \\
            &= q^{\frac{1}{2}k \cdot l}\delta'_{\eps,-k'}(\lambda + \rho_{k'})P_{\lambda - \tilde{\rho}_l}(k + l^\wedge),
        \end{aligned}
    \]
    using the $q$-shift principle (Proposition \ref{prop:shift_prin_q}) for the second equality.
\end{proof}

The transmutation property of $G^{(\eps)}_+(k)$ is an immediate consequence of Theorem \ref{thm:sym_shift_factor_q}, cf.~the proof of Corollary \ref{cor:sym_transmutation_prop}.

\begin{corollary}\label{cor:sym_transmutation_prop_q}
    The operator $G^{(\eps)}_+(k)$ satisfies the transmutation property
    \[
        G^{(\eps)}_+(k)Y^{f'}(k) = Y^{f'}(k + l^\wedge)G^{(\eps)}_+(k), \qquad f' \in A_0',
    \]
    viewed as operators from $A_0(k)$ to $A_0(k + l^\wedge)$.
\end{corollary}


\subsection{Non-symmetric $q$-shift operators}\label{sec:non_sym_shift_operators_12_q}
Following the construction of the non-symmetric shift operators in the Heckman--Opdam setting, it is necessary to introduce a trigonometric analog of the polynomial from Lemma \ref{lem:minimal_poly} which can be evaluated in Cherednik operators.

\begin{lemma}\label{lem:minimal_poly_q}
    Writing $\varpi' = \prod_{a' \in S_{01}'^+}\delta_{a',0}$, there exists a trigonometric polynomial $\frakq \in A' \otimes A'$ such that
    \begin{equation}\label{eq:minimal_poly_prop_q}
        \frakq(x, wx) = \varpi'(x)\delta_{e,w}, \qquad w \in W_0,
    \end{equation}
    for all regular $x \in E$.
\end{lemma}

\begin{proof}
    In light of the proof of Lemma \ref{lem:minimal_poly}, we seek for an appropriate subspace $H' \subset A' = K[L']$ that is isomorphic to $K[W_0]$ as a vector space and such that, for regular $x \in E$, a linear isomorphism is given by
    \[
        U(x) : H' \to K[W_0], \thinspace h \mapsto \sum_{w\in W_0}h^w(x)\delta_w.
    \]
    To establish the existence of such an $H'$, we use the fact that $A'$ is a free $A'_0$-module of rank $|W_0|$, which is a corollary of the Pittie--Steinberg theorem, see \cite[Thm.2.2.]{Ste75}. Moreover, this theorem provides an explicit basis $(\fraku_w)_{w \in W_0}$ of $A'$, called the Steinberg basis, which we use to construct the subspace $H' = \text{Span}_K\{\fraku_w \mid w \in W_0\} \subset A'$.
    
    Let us give the construction of this basis. We observe that $L'$ is the weight lattice of the root system $S'_{02} = S'_0 \cap S'_2$. The Weyl group of this root system is $W_0$ and a base of simple roots is given by $b'_i = u_{a'_i}a'_i$ for $i \in I_0$. We denote by $(\lambda_i)_{i \in I_0}$ the fundamental weights of $L'$ that are dual to the simple roots $(b'_i)_{i \in I_0}$. For each $w \in W_0$, the basis element $\fraku_w$ is defined by 
    \[
        \fraku_w = e^{w^{-1}\lambda_w}, \qquad \lambda_w = \sum_{i \in I_0 \thinspace : \thinspace w^{-1}b'_i \in S_{02}'^-}\lambda_i.
    \]
    We represent the linear isomorphism $U(x)$ by the matrix $\hat{U}(x) = (\fraku_w^v(x))_{w,v \in W_0}$ with respect to the indicated bases. Note that $\varpi'$ coincides with the Weyl denominator for $S'_{02}$, i.e., $\varpi' = \prod_{b' \in S_{02}'^+}(e^{b'/2} - e^{-b'/2})$. As shown in \cite[p.175]{Ste75}\footnote{In the notation of \cite{Ste75}, we note that the expression for $D_1$ on page 175 should read $\prod(a^{1/2} - a^{-1/2})^{n_a}$ ($a \in \Sigma^+$), and we observe that in our case, i.e., when $W'$ is trivial, the number of rows $n_a$ interchanged by $w_a$ equals $|W|/2$.}, the determinant of $U(x)$ is equal to $\varpi'(x)^{d/2}$,\vspace{-0.1cm} where $d = |W_0|$. The same argument shows that the minors of $\hat{U}(x)$ are divisible by $\varpi'(x)^{d/2 - 1}$. Thus, the entries of $\varpi'(x)\hat{U}(x)^{-1}$ belong to $A'$.
    
    As in the proof of Lemma \ref{lem:minimal_poly}, the desired trigonometric polynomial $\frakq$ satisfying \eqref{eq:minimal_poly_prop_q} is defined by $\frakq(x, \cdot) = \varpi'(x)U(x)^{-1}(\delta_e) \in H'$. Explicitly,
    \begin{equation}\label{eq:minimal_poly_exp_q}
        \frakq = \sum_{w \in W_0} \frakq_w \otimes \fraku_w,
    \end{equation}
    where $(\frakq_w(x))_{w \in W_0}$ is the first column of $\pi(x)\hat{U}(x)^{-1}$.
\end{proof}

Next, we define the operator $Q_\mu(k) = \frakq(r_{k'}(\mu),Y(k))$ for each $\mu \in L$, where $\frakq$ is the trigonometric polynomial from Lemma \ref{lem:minimal_poly_q}, i.e,
\[
    Q_\mu(k) = \sum_{w \in W_0} \frakq_w(r_{k'}(\mu))Y^{\fraku_w}(k)\vspace{-0.1cm}
\]
in terms of the expansion \eqref{eq:minimal_poly_exp_q} of $\frakq$.

As in Corollary \ref{cor:proj_operator_eHO}, we are interested in computing the action of the operator $Q_\mu(k)$ on the polynomials $P^{(\eps)}_\lambda(k)$.

\begin{lemma}\label{lem:proj_operator_eMK}
    Let $\lambda \in L_+$ and $\mu \in W_0\lambda$. The operator $Q_\mu(k)$ defines a map from $A_\lambda(k)$ to $A_\lambda^\mu(k)$. Moreover, if $\lambda \in \tilde{\rho}_l + L_+$, then
    \[
        Q_\mu(k)P^{(\eps)}_\lambda(k) = \eps(\overline{v}(\mu))\eps(\overline{v}(\mu))\tau_{\overline{v}(\mu)w_{0\lambda},k}^{-1}\tau_{w_0,k}\varpi'(r_{k'}(\mu))\bfc_{-\eps k'}^+(\overline{v}(\mu)w_{0\lambda})(\lambda + \rho_{k'})E_\mu(k).
    \]
\end{lemma}

\begin{proof}
    Similar to the proof of Lemma \ref{lem:proj_operator}, we find that $Q_\mu(k)E_\nu(k) = \delta_{\mu,\nu}\varpi'(r_{k'}(\mu))$ for all $\nu \in W_0\mu$. Therefore, the action of $Q_\mu(k)$ on $P^{(\eps)}_\lambda(k)$ can be recovered from the coefficients from Lemma \ref{lem:eMK_in_nMK_basis}.
\end{proof}

\begin{example}\label{exa:minimal_poly_q}\leavevmode
    \begin{enumerate}[(i)]
        \item For the weight lattice $L'$ of the root system $R = \{\pm\alpha_1\}$ of type $A_1$, the trigonometric polynomial $\frakq$ equals $e^{\alpha_1} \otimes 1 - 1 \otimes e^{-\alpha_1}$, so that $Q_\mu(k) = \big(q^{<r_{k'}(\mu),\alpha_1^\vee>} - Y^{-\alpha_1}(k)\big)$.
        \item Let $L'$ be the weight lattice of the root system $R = \{\epsilon_i - \epsilon_j \mid i, j = 1,2,3, \thinspace i \neq j\}$ of type $A_2$. Writing $\varpi_1$ and $\varpi_2$ for the fundamental weights corresponding to the simple roots $\alpha_1 = \epsilon_1 - \epsilon_2$ and $\alpha_2 = \epsilon_2 - \epsilon_3$, the Steinberg basis is given by\vspace{-0.1cm}
        \begin{align*}
            &\qquad \fraku_1 = 1, \quad \fraku_2 = e^{-\varpi_1}, \quad \fraku_3 = e^{-\varpi_2}; \\
            &\fraku_4 = e^{-\varpi_1 + \varpi_2}, \quad \fraku_5 = e^{\varpi_1 - \varpi_2}, \quad \fraku_6 = e^{-\varpi_1 - \varpi_2}.\vspace{-0.1cm}
        \end{align*}
        In terms of this basis, the trigonometric polynomial takes the following form:\vspace{-0.1cm}
        \[
             \frakq = \sum_{j = 1}^3\fraku_{7-j}^* \otimes \fraku_j - \sum_{j = 4}^6\fraku_{7-j}^* \otimes \fraku_j.\vspace{-0.1cm}
        \]
        Here $\ast$ denotes the involution defined by $(e^\lambda)^* = e^{-\lambda}$ for $\lambda \in L'$. 
    \end{enumerate}
\end{example}

In order to provide uniform notation, it becomes necessary to rephrase the $q$-shift principle (Proposition \ref{prop:shift_prin_q}) as follows:

\begin{corollary}\label{cor:shift_prin_q_v2}
    Let $\lambda \in L_+$. If $\lambda \mp \tilde{\rho}_l \in L_+$, then
    \[
        \Delta_{\eps,k}^\mp P^{(\eps_\pm)}_\lambda(k) = q^{n_{\eps,k}^\pm}P^{(\eps_\mp)}_{\lambda \mp \tilde{\rho}_l}(k \pm l^\wedge),
    \]
    with $\Delta_{\eps,k}^- = \delta_{\eps,k}^{-1}$, $\Delta_{\eps,k}^+ = \delta_{\eps,k - l^\wedge}$, $n_{\eps,k}^+ = \frac{1}{2}k \cdot l$, and $n_{\eps,k}^- = -\frac{1}{2}(k - l^\wedge) \cdot l$.
\end{corollary}

At this point, we can define the $\eps$-forward/backward non-symmetric $q$-shift operator for fixed weights $\lambda \in L_+$ and $\mu \in W_0\lambda$:
\begin{equation}\label{eq:fw_shift_operator}
    \calG^{(\eps)}_\pm(\mu, k) = Q_{\mu_{\eps,\pm}}(k \pm l^\wedge) \circ \Delta_{\eps,k}^\mp \circ U_{\eps_\pm}(k) : A_\lambda^\mu(k) \to A_{\lambda \mp \tilde{\rho}_l}^{\mu_{\eps,\pm}}(k \pm l^\wedge),
\end{equation}
with $\mu_{\eps,\pm} = \overline{v}(\mu)w_{0\lambda}(\lambda \mp \tilde{\rho}_l)$.

We continue by computing the shift factor of the operator $\calG^{(\eps)}_\pm(\mu, k)$. For this, we first prove the analog of Vanishing Lemma 1 (Lemma \ref{lem:vanishing_1}).

\begin{lemma}[$q$-Vanishing Lemma 1]\label{lem:vanishing_1_q}
    Let $\lambda \in L_+$ and $\mu \in W_0\lambda$. If $\lambda - \tilde{\rho}_l \notin L_+$, then
    \[
        \bfc_{\eps k'}^-(\overline{v}(\mu)w_{0\lambda})(\lambda + \rho_{k'}) = 0.
    \]
\end{lemma}

\begin{proof}
    Since $S_{01}' = R^\vee$ and $W_0$ is its Weyl group, there exists an $i \in I_0$ such that the coroot $\alpha_i^\vee \in R^{\vee+} \cap (\overline{v}(\mu)w_{0\lambda})^{-1}R^{\vee-}$ satisfies $\alpha_i^\vee \in R^\vee_\lambda$ and $\eps(r_i) = -1$ by the proof of Lemma \ref{lem:vanishing_1_q}. Equivalently, we have that $a_i' = \alpha_i^\vee \in S_1'^+ \cap (\overline{v}(\mu)w_{0\lambda})^{-1}S_1'^-$. We observe that
    \[
        (1 - \tau_{a_i',\eps k'}\tilde{\tau}_{a_i',\eps k'}q^{<a_i',\lambda + \rho_{k'}>}) = 1 - q^{<a_i',\lambda + \rho_{k'}> \thinspace + \thinspace \eps(r_i)k'(a'_i)} = 0,
    \]
    since $<a_i',\lambda + \rho_{k'}> + \thinspace \eps(r_i)k'(a'_i) = \thinspace <a_i',\rho_{k'}> - k'(a'_i) = 0$. Thus, $\bfc_{\eps k'}^-(\overline{v}(\mu)w_{0\lambda})(\lambda + \rho_{k'}) = 0$ because one of its factors vanishes.
\end{proof}

\begin{proposition}\label{prop:fw_shift_factor_q}
    Let $\lambda \in L_+$ and $\mu \in W_0\lambda$. Then
    \[
        \calG^{(\eps)}_\pm(\mu,k)E_\mu(k) = H^{(\eps)}_\pm(\mu,k)E_{\mu_{\eps,\pm}}(k \pm l^\wedge),
    \]
    with
    \[
        \begin{aligned}
            H^{(\eps)}_\pm(\mu, k) &= \eps_\pm(\overline{v}(\mu))\eps_\mp(\overline{v}(\mu_{\eps,\pm}))\tau_{\overline{v}(\mu)w_{0\lambda},k}\tau_{\overline{v}(\mu_{\eps,\pm})w_{0(\lambda \mp \tilde{\rho}_l)},k \pm l^\wedge}^{-1}\tau_{w_0,l^\wedge}^\pm q^{n_{\eps,k}^\pm}\varpi'(r_{k' \pm {l^\wedge}'}(\mu_{\eps,\pm})) \\
            &\qquad \times \bfc_{\eps_\pm k'}^-(\overline{v}(\mu)w_{0\lambda})(\lambda + \rho_{k'}) \bfc_{-\eps_\mp(k' \pm {l^\wedge}')}^+(\overline{v}(\mu_{\eps,\pm})w_{0(\lambda \mp \tilde{\rho}_l)})(\lambda + \rho_{k'}).
        \end{aligned}
    \]
\end{proposition}

\begin{proof}
    Arguing as in the proof of Proposition \ref{prop:fw_shift_factor}, the shift factor is computed by tracing the maps in accordance with Lemma \ref{lem:relating_U_nMK_and_eMK}, Lemma \ref{lem:proj_operator_eMK}, and Corollary \ref{cor:shift_prin_q_v2} in case that $\lambda \in \tilde{\rho}_l + L_+$. For $\lambda - \tilde{\rho}_l \notin L_+$, we have that $\calG^{(\eps)}_\pm(\mu,k)E_\mu(k) = 0$ since then $U_\eps(k)E_\mu(k) = 0$. The proposition holds in this case by virtue of Lemma \ref{lem:vanishing_1_q}, which implies that $H^{(\eps)}_+(\mu, k) = 0$.
\end{proof}

Proceeding as before, to eliminate the $\mu$-dependence, we rewrite the expression \eqref{eq:fw_shift_operator} of $\calG^{(\eps)}_\pm(\mu,k)$, and replace $Q_{\mu_{\eps,\pm}}(k \pm l^\wedge)$ by operators $\calQ^{(\eps)}_\pm(w, k)$ that are independent of $\mu$. More specifically, for each linear character $\eps$, we define the $\eps$-forward/backward non-symmetric $q$-shift operator
\[
    \calG^{(\eps)}_\pm(k) = \sum_{w\in W_0}\eps_\pm(w)\calQ^{(\eps)}_\pm(w,k) \Delta_{\eps,k}^\mp w,
\]
where
\[
    \calQ^{(\eps)}_\pm(w, k) = \sum_{w \in W_0} Y^{\fraku_w}(k \pm l^\wedge)C_{\Delta_{\eps,k}^\mp w}(Y^{\frakq_w}(k)).
\]

\begin{proposition}\label{prop:def_shift_operator_q}
    The operator $\calG^{(\eps)}_\pm(k)$ defines a map from $A(k)$ to $A(k \pm l^\wedge)$, and its restriction to $A^\mu_\lambda(k)$, with $\lambda \in L_+$ and $\mu \in W_0\lambda$, is given by
    \[
        \calG^{(\eps)}_\pm(k)\big|_{A_\lambda^\mu(k)} =
        \begin{cases}
            \calG^{(\eps)}_\pm(\mu,k)\big|_{A_\lambda^\mu(k)} & \text{if } v(\mu) = v(\mu_{\eps,\pm}); \\
            0 & \text{otherwise}.
        \end{cases}
    \]
\end{proposition}

\begin{proof}
    The proposition is the analog of combining Lemma \ref{lem:proj_operator_replacement} and Proposition \ref{prop:def_shift_operator}, and their proofs carry over mutatis mutandis.
\end{proof}

To relate the shift factor from Proposition \ref{prop:fw_shift_factor_q}, we establish the analog of Vanishing Lemma 2 (Lemma \ref{lem:vanishing_2}).

\begin{lemma}[$q$-Vanishing Lemma 2]\label{lem:vanishing_2_q}
    Let $\lambda \in L_+$ and $\mu \in W_0\lambda$. If $v(\mu) \neq v(\mu_{\eps,\pm})$, then
    \[
        \bfc_{-\eps_\mp(k' \pm {l^\wedge}')}^+(\overline{v}(\mu)w_{0\lambda})(\lambda + \rho_{k'}) = 0.
    \]
\end{lemma}

\begin{proof}
    The proof of Lemma \ref{lem:vanishing_2} can be adjusted to function for the lemma as well, cf.~the proof of Lemma \ref{lem:vanishing_1_q}.
\end{proof}

This following corollary is immediate.

\begin{corollary}\label{cor:adj_shift_factor_q}
    Let $\lambda \in L_+$ and $\mu \in W_0\lambda$. Writing
    \[
        \begin{aligned}
            \calH^{(\eps)}_\pm(\mu, k) &= \eps(\overline{v}(\mu))\tau_{\overline{v}(\mu),l^\wedge}^\mp\tau_{w_0,l^\wedge}^\pm q^{n_{\eps,k}^\pm}\varpi'(r_{k'}(\mu))\bfc_{\eps_\pm k'}^-(\overline{v}(\mu)w_{0\lambda})(\lambda + \rho_{k'}) \\
            &\qquad \times \bfc_{-\eps_\mp(k' \pm {l^\wedge}')}^+(\overline{v}(\mu)w_{0\lambda})(\lambda + \rho_{k'}),
        \end{aligned}
    \]
    we have
    \[
        \calH^{(\eps)}_\pm(\mu, k) =
        \begin{cases}
            H^{(\eps)}_\pm(\mu, k) & \text{if } v(\mu) = v(\mu_{\eps,\pm}); \\
            0 & \text{otherwise}.
        \end{cases}
    \]
\end{corollary}

We proceed to prove the transmutation property of the operator $\calG^{(\eps)}_\pm(k)$.

\begin{theorem}\label{thm:transmutation_prop_q}
    The operator $\calG^{(\eps)}_\pm(k) : A(k) \to A(k \pm l^\wedge)$ is a non-symmetric $q$-shift operator with shift $\pm l^\wedge$, i.e., it satisfies the transmutation property
    \[
        \calG^{(\eps)}_\pm(k) Y^{f'}(k) = Y^{f'}(k \pm l^\wedge) \calG^{(\eps)}_\pm(k), \qquad f' \in A'.
    \]
    Furthermore, $\calG^{(\eps)}_\pm(k)$ shifts the multiplicity of the non-symmetric Macdonald--Koornwinder polynomials by $\pm l^\wedge$:
    \[
         \calG^{(\eps)}_\pm(k)E_\mu(k) = \calH^{(\eps)}_\pm(\mu,k)E_{\mu_{\eps,\pm}}(k \pm l^\wedge), \qquad \mu \in L.
    \]
\end{theorem}

\begin{proof}
    Adapt the proof of Theorem \ref{thm:transmutation_prop} using Proposition \ref{prop:def_shift_operator_q} and Corollary \ref{cor:adj_shift_factor_q}.
\end{proof}

\begin{remark}\label{rmk:rank_one_q}
    For the third case in rank one, i.e., when $S = S'$ is of type $(C_1^\vee,C_1)$, and the sign character $\eps$, the $\eps$-forward/backward non-symmetric $q$-shift operators have shifts $\pm(1,1,0,0)$. In this case, the (non-)symmetric Macdonald--Koornwinder polynomials can be identified with the (non-)symmetric Askey--Wilson polynomials, up to normalization. In \cite{AsWi85}, Askey and Wilson found forward/backward $q$-shift operators with shift $\pm\frac{1}{2}(1,1,1,1)$ for their polynomials, see \cite[\S6.4--6]{Mac03} or \cite{NoSt04} for an affine Hecke algebraic approach. In addition, contiguity-type $q$-shift operators were discovered by Kalnins and Miller \cite{KaMi89} with shifts given by the six permutations of $\frac{1}{2}(1,1,-1,-1)$. The non-symmetric analogs of these shift operators were studied in \cite{vHS25}, and our $\eps$-forward/backward non-symmetric $q$-shift operators most likely decompose into the non-symmetric $q$-shift operators with shifts $\pm\frac{1}{2}(1,1,1,1)$ and $\pm\frac{1}{2}(1,1,-1,-1)$.
\end{remark}

Finally, we discuss the analog of Proposition \ref{prop:restriction}. To state it properly, it is necessary to adjust the shift factors in case that the character differs from the sign character. This leads us to formulate the following conjecture, cf.~Conjecture \ref{conj:fundamental_shift_operator}.

\begin{conjecture}\label{conj:fundamental_shift_operator_q}
    For any linear character $\eps$ of $W_0$, there exists a non-symmetric $q$-shift operator $\tilde{\calG}^{(\eps)}_\pm(k)$ with shift $\pm l$ such that
    \[
         \tilde{\calG}^{(\eps)}_\pm(k)E_\mu(k) = \tilde{\calH}^{(\eps)}_\pm(\mu,k)E_{\mu_{\eps,\pm}}(k \pm l^\wedge), \qquad \mu \in L,
    \]
    with
    \[
        \begin{aligned}
            \tilde{\calH}^{(\eps)}_\pm(\mu, k) &= \eps(\overline{v}(\mu))\tau_{\overline{v}(\mu),l^\wedge}^\mp\tau_{w_0,l^\wedge}^\pm q^{n_{\eps,k}^\pm}\delta'_\eps(r_{k'}(\mu))\bfc_{\eps_\pm k'}^-(\overline{v}(\mu)w_{0\lambda})(\lambda + \rho_{k'}) \\
            &\qquad \times \bfc_{-\eps_\mp(k' \pm {l^\wedge}')}^+(\overline{v}(\mu)w_{0\lambda})(\lambda + \rho_{k'}).
        \end{aligned}
    \]
    The shift factor $\tilde{\calH}^{(\eps)}_\pm(\mu, k)$ is obtained from $\calH^{(\eps)}_\pm(\mu, k)$ by substituting $\delta'_\eps = \delta'_{\eps,0}$ for $\varpi'$.
\end{conjecture}

Conjecture \ref{conj:fundamental_shift_operator_q} is valid for the sign character $\eps$, since then we can simply take $\tilde{\calG}^{(\eps)}_\pm(k) = \calG^{(\eps)}_\pm(k)$.

\begin{proposition}\label{prop:restriction_q}
    On the assumption of Conjecture \ref{conj:fundamental_shift_operator_q}, noting its validity for the sign character, the operator $\tilde{\calG}^{(\eps)}_+(k)$ restricts to a difference operator on $A_0(k)$. This difference operator coincides with $G^{(\eps)}_+(k)$ from Theorem \ref{thm:sym_shift_factor_q}, so that
    \[
        \tilde{\calG}^{(\eps)}_+(k)P_\lambda(k) = h^{(\eps)}_+(\lambda, k) P_{\lambda - \tilde{\rho}_l}(k + l^\wedge), \qquad \lambda \in L_+.
    \]
\end{proposition}

\begin{proof}
    Expanding the definitions of the $\bfc$-functions and observing that $\delta'_\eps(r_{k'}(\mu))$ cancels the denominators, we arrive at
    \begin{equation}\label{eq:restriction_q_step}
        \tilde{\calH}^{(\eps)}_+(\mu,k) = \tau_{\overline{v}(\mu),l^\wedge}^{-1}\tau_{w_0,l^\wedge} q^{\frac{1}{2}k \cdot l}\delta'_{\eps,-(k' + m(\overline{v}(\mu)w_{0\lambda}){l^\wedge}')}(\lambda + \rho_{k'})
    \end{equation}
    for all $\lambda \in L_+$ and $\mu \in W_0\lambda$, cf.~Lemma \ref{lem:shift_factor_v2}. The proof of Proposition \ref{prop:restriction} carries over since $P_\lambda(k)$ can be expressed as a linear combination of the polynomials $E_\mu(k)$ with $\mu \in W_0\lambda$ by Lemma~\ref{lem:eMK_in_nMK_basis}. Hence, it suffices to prove the analog of the relation \eqref{eq:restriction_step} between the coefficients and the shift factors. In this case, we have
    \[
        \begin{aligned}
            &\tilde{\calH}^{(\eps)}_+(\mu,k) \tau_{\overline{v}(\mu)w_{0\lambda},k}^{-1}\tau_{w_0,k} c^+_{-k'}(\overline{v}(\mu)w_{0\lambda})(\lambda + \rho_{k'}) \\
            &\qquad = h^{(\eps)}_+(\lambda,k)\tau_{\overline{v}(\mu_{\eps,+})w_{0(\lambda - \tilde{\rho}_l)},k + l^\wedge}^{-1}\tau_{w_0,k + l^\wedge}c^+_{-(k' + {l^\wedge}')}(\overline{v}(\mu_{\eps,+})w_{0(\lambda - \tilde{\rho}_l)})(\lambda + \rho_{k'})
        \end{aligned}
    \]
    for all $\mu \in W_0\lambda$, which follows from \eqref{eq:restriction_q_step} when $\lambda \in \tilde{\rho}_l + L_+$ and $v(\mu) = v(\mu_{\eps,\pm})$. The relation remains valid if either $\lambda \notin \tilde{\rho}_l + L_+$ or $v(\mu) \neq v(\mu_{\eps,\pm})$, which can be shown using the vanishing arguments from the proof of the Proposition \ref{prop:restriction}.
\end{proof}


\subsection*{Acknowledgments}
The research of M.v.H. was supported by the Research Foundation Flanders (FWO) grant 1139226N.


\end{document}